\documentclass[a4paper,11pt]{article}

\usepackage{stmaryrd}
 \usepackage{braket}
\usepackage[T1]{fontenc}
\usepackage[utf8]{inputenc}
\usepackage{lipsum}
\usepackage{amsmath, amssymb, mathtools, amsthm}
\usepackage{comment} 
\usepackage{crossreftools}
\usepackage[usenames,dvipsnames]{xcolor}
\usepackage{graphicx} 
\usepackage{tikz}
\usepackage{enumerate} 
\usepackage[margin=0.96in]{geometry}
\usepackage{bbm}

\usepackage[numbers,sort&compress]{natbib}
\usepackage[colorlinks=true]{hyperref}
\hypersetup{urlcolor=black, citecolor=black, linkcolor=black}
\usepackage{mathtools}
\usepackage{float}
\usepackage{xspace}
\usepackage{wrapfig}

 \usepackage{epsfig}
 \usepackage{subcaption}
 \usepackage{multirow}
 \usepackage{lineno}
 \usepackage{fullpage}
 \usepackage[normalem]{ulem} 
 \usepackage{makeidx}
 \usepackage{soul}
 
\newcommand{\OmegaT}{{\Omega_T}} 
\newcommand{\curlyA}{\mathcal{A}}
\newcommand{\curlyB}{\mathcal{B}}

\newcommand{\curlyD}{\mathcal{D}}

\newcommand{\curlyH}{\mathcal{H}}

\newcommand{\curlyR}{\mathcal{R}}

\usepackage{cancel}
\usepackage[title]{appendix}
\usepackage{dsfont}

\DeclarePairedDelimiter{\prt}{(}{)}

\newcommand \commentout[1] {}

\DeclareMathOperator*{\supp}{\operatorname{supp}}

\newcommand{\partialt}[1]{\dfrac{\partial#1}{\partial t}}

\DeclareMathAlphabet{\mathup}{OT1}{\familydefault}{m}{n}
\newcommand{\dx}[1]{\mathop{}\!\mathup{d} #1}

\theoremstyle{plain}
\newtheorem{theorem}{Theorem}[section]
\newtheorem{lemma}[theorem]{Lemma}

\newtheorem{corollary}[theorem]{Corollary}

\theoremstyle{remark}
\newtheorem{remark}[theorem]{\bf Remark}

\newcommand{\ie}{\emph{i.e.}\;}
\newcommand{\cf}{\emph{cf.}\;}

\newcommand{\sign}{\mathrm{sign}}

\newcommand{\ddt}{\frac{\dx{}}{\dx{t}}}

\newcommand{\es}{\varepsilon}
\newcommand{\R}{\mathbb{R}}

\begin{document}

\title{Phenotypic heterogeneity in a model of tumor growth: existence of solutions and incompressible limit
}

\author{Noemi David\thanks{Institut Camille Jordan, Universit\'e Claude Bernard - Lyon 1, 21 Av. C. Bernard, 69100 Villeurbanne, France} }
\definecolor{green}{rgb}{0.2, 0.9, 0.5}
\maketitle
\begin{abstract}
We consider a (degenerate) cross-diffusion model of tumor growth structured by phenotypic trait.
We prove the existence of weak solutions and the incompressible limit as the pressure becomes stiff extending methods recently introduced in the context of two-species cross-diffusion systems. In the stiff-pressure limit, the compressible model generates a free boundary problem of Hele-Shaw type. Moreover, we prove a new $L^4$-bound on the pressure gradient.
\end{abstract} 

\begin{flushleft}
    \noindent{\makebox[1in]\hrulefill}
\end{flushleft}
\begin{flushright}
    \noindent{\makebox[1in]\hrulefill}
\end{flushright} 

\section{Introduction}

We consider the following model of tumor growth structured by phenotypic trait, represented by the continuous variable $y\in[0,1]$. The cell proliferation rate depends on both the trait and the pressure inside the tissue. The motion of cells is driven by Darcy's law, since the cell movement is passively generated by the birth and death of cells which create pressure gradients. We denote by $n=n(y,x,t)$ the density of the population with phenotypic trait $y\in[0,1]$, and with $\varrho=\varrho(x,t)$ the total density at point $x\in\R^d$ and time $t>0$.
The pressure is related to the total density by the following power law
 \begin{equation}\label{pressure law pheno}
     p(x,t)= \prt*{\varrho(x,t)}^\gamma, \qquad \gamma>1.
 \end{equation}
The model is the following
 \begin{equation}\label{eq: prob pheno}
    \begin{dcases} \partialt n(y,x,t) - \nabla \cdot(n(y,x,t) \nabla p(x,t)) = n(y,x,t) R(y,p(x,t)), \qquad (y,x,t) \in [0,1]\times\R^d \times (0,\infty),\\[0.2em]
    \varrho(x,t) = \int_0^1 n(y,x,t) \dx{y},
    \end{dcases}
 \end{equation}
with initial data $n_0(y,x)\in L_+^{\infty}([0,1]\times \R^d)\cap L^1([0,1]\times \R^d)$, and where $\nabla$ and $\Delta$ are derivatives with respect to the variable $x$.

Let us point out that the equation satisfied by $\varrho(x,t)$ is a porous medium-reaction equation with coefficient $\gamma+1$, namely
\begin{equation}\label{eq: density rho}
     \partial_t \varrho - \frac{\gamma}{\gamma+1}\Delta \varrho^{\gamma+1} = \varrho \curlyR,
 \qquad 
    \curlyR=\int_0^1 \sigma(y) R(y,p)\dx{y},
\end{equation}
where with $\sigma=n/\varrho$ we denote the phenotype density fractions, while $\curlyR$ represents the total population growth rate.

\bigskip
\noindent
\textbf{Structured models: motivations.} 
The mathematical modelling of living tissue has attracted increasing attention in the last decades for both its ability to describe and investigate biological phenomenon and the extremely challenging mathematical problems that arise from such models. Among them, there is a growing interest towards models where the population density is structured by a phenotypic trait.
In structured models, intra-population heterogeneity is taken into account by letting the mobility rate and/or the growth rate of each phenotypic distribution be functions of the structuring variable. Most of these models are based on Fisher-KPP equations, hence they describe the random movement of the cells through a linear diffusion term, with a phenotype-dependent mobility rate, and cell proliferation through a logistic growth rate. Non-local reaction terms are also considered, as in the non-local version of the Fisher-KPP model, \cite{perthamefkppnonlinear}, as well as divergence terms with respect to the phenotypic state to account for mutations, see for instance \cite{Calvez}. In this paper, Calvez \textit{et al.} introduce a model in which only the mobility rate depends on the phenotypic trait. In particular, they assume the mobility rate to be proportional to the structuring variable. 
Computing an exact asymptotic traveling wave solution, they show that phenotypic segregation occurs and leads to front acceleration. Originating from \cite{Calvez}, the acceleration of invasion fronts has been further studied in \cite{BMR, BouinCalvezPerthame} in the case of unbounded mobility, see also \cite{Lorenziesaim, Gatenbyanderson, andersonbyrnelorenzi2020} and references therein for applications of structured PDEs models to tumor growth.

In \cite{lorenziperthameruan2021}, Lorenzi \textit{et al.} propose a model structured by phenotypic trait to study a phenomenon arising in cancer development which is usually referred to as `\textit{growth or go}', namely the dichotomy of migration and proliferation. As investigated in \cite{Costmigration, GerleeA2009, GerleeN2012, Dichotomy}, more mobile cells tend to divide less than cells that have a lower mobility rate. For this reason, the authors consider mobility and growth rates which are, respectively, increasing and decreasing functions of the structuring variable. Unlike the previously mentioned models, they consider a velocity field which depends on the total population, \ie the integral of the distributions with respect to the phenotypic trait. In particular, they take the velocity field to be proportional to the gradient of the total density. Therefore, the diffusion in the model is degenerate and no longer linear. The authors study the creation of compactly supported invasion fronts, and show that phenotypic separation occurs in the case of bounded mobility while the front undergoes acceleration in the case of unbounded mobility. We also refer the reader to \cite{Lorenzi2022} for an extension of this result to more general pressure laws and the derivation of a corresponding individual-based model.

\paragraph{Porous medium models.}
As suggested in \cite{lorenziperthameruan2021}, a natural generalisation of their model consists of considering a pressure $p$ related to the density by a power law with exponent greater than 1, as in Eq.~\eqref{pressure law pheno}. This pressure law has been extensively used in the modelling of tumor growth since it can be associated with the pressure of a compressible fluid. Combining the power law with Darcy's law yields to porous medium-type equations as Eq.\eqref{eq: density rho}. Indeed, the invasion of cancer cells can be seen as the motion of a fluid through a porous medium (the extra-cellular matrix)~\cite{byrneChaplain96}.

The power law was first adopted for one-species models of tumor growth, see for instance \cite{PQV, PV2015} and references therein. Furthermore, this pressure law is of particular interest since passing to the limit $\gamma\to\infty$, it is possible to establish a link between compressible models and "geometrical" problems. As the pressure becomes stiffer and stiffer, porous medium models converge to Hele-Shaw free boundary problems where the density is saturated and the pressure satisfies an elliptic equation. This limit, referred to as \textit{incompressible limit} or \textit{stiff pressure limit}, has been studied for several non-structured one-species models, starting from the seminal paper by Perthame \textit{et al.} \cite{PQV}. For an overview on the single-species case, we refer the reader to \cite{PQV, DP, DavidSchmidtchen2021, KP17, DavidDebiecPerthame2021, DeSc, GKM2020, AKY14, activemotion} and references therein.

\paragraph{Multi-species extensions.} Lately, multi-phase extensions of the model introduced in \cite{PQV} have been studied from different perspectives. Multi-species models allow to study the interaction between different types of tissue, for instance, cancer tissue, immune cells, healthy tissue, or dead tissue. In cross-reaction-diffusion systems, the coupling of the single densities equations gives rise to new mathematical challenges, such as the loss of regularity due to internal layers, namely regions where two species get in contact. For this reason, the mathematical analysis of these models presents many involved open problems. 
In 2018, Carrillo \textit{et al.} show the existence of solutions to a reaction-cross-diffusion system of two equations using methods from optimal transport \cite{CFSS}. Their result, which was achieved in one spatial dimension, was later extended in 2019 by Gwiazda \textit{et al.} in multiple dimensions \cite{GPS}. Here, the authors consider a two-species system that is the discrete (and cross-reaction) counterpart of our model, \ie Eq.~\eqref{eq: prob pheno} for $y\in \{1,2\}$. In particular, the two species evolve under Darcy's law, where the pressure is given by $p=(n_1+n_2)^\gamma$, and $n_i, \ i=1,2$ denotes the two phases. Their existence result relies on applying a uniformly parabolic regularisation to the initial data and then passing to the limit. To this end, the most involved term is the nonlinear cross-diffusion term $n_i\nabla p$. In order to pass to the limit, the authors prove an $L^2$-version of the Aronson-B\'enilan estimate, which is a celebrated estimate in the context of porous medium equations and provides a bound on the Laplacian of the pressure. We refer the reader to \cite{AB} for the classical result.
The same problem was then approached in \cite{price2020global}, in which the authors are able to prove convergence by focusing on the quantity $(n_1+n_2)^{\gamma+1}$ rather than the pressure itself. Their proof is simpler since it does not require any regularity result on the second-order derivatives of $p$. In fact, in \cite{price2020global} the authors recover the strong convergence of $\nabla (n_1+n_2)^{\gamma+1}$ without using the Aronson-B\'enilan estimate of \cite{GPS}, for which a restrictive condition on the reaction terms was needed. A recent result by Jacobs \cite{jacobs} uses a similar energy dissipation relation in order to obtain strong compactness of $\nabla p$ in a very general setting.
We refer the reader to \cite{hilhorst2012} for a reference of the existence of smooth solutions in the case of smooth initial data.
Moreover, let us mention that, to the best of the author's knowledge, up to date no uniqueness results for the cross-diffusion model analysed in \cite{price2020global, GPS} are known. This question surely represents a challenging open problem.
 
As mentioned above, the analysis of the incompressible limit for porous medium models has a long history and has been addressed by many researchers for several models. The stiff limit for systems including two different species has been first addressed by Bubba \textit{et al.} in 2019, \cite{BPS2020}, where the authors use an approach based on a $L^2$-Aronson-B\'enilan estimate in the spirit of \cite{GPS}. However, due to the absence of $BV$ controls on the single species population densities, their argument only works in dimension 1. The result in any spatial dimension has been recently achieved by Liu and Xu in \cite{LX2021}, where the authors consider a cross-reaction-diffusion system in a bounded domain with Neumann boundary conditions. Rather than dealing with the pressure, $p_\gamma= \varrho_\gamma^\gamma$, the authors focus on the quantity $\varrho_\gamma^{\gamma+1},$ proving strong compactness of its gradient, thus being able to prove convergence of the cross-diffusion terms. However, they are not able to include pressure-dependent reaction terms, and proving strong compactness of the pressure itself remains an open question in this setting. A similar method to obtain strong compactness on the pressure gradient without needing any uniform bound on the second-order derivatives was also developed in \cite{jacobs}. Recently, the same author proved that it is indeed possible to show uniform boundedness in $L^1$ of the pressure time derivative, hence finding strong compactness of the gradient through a newer version of the Aronson-B\'enilan estimate \cite{jacobs2}. The stiff limit for cross-diffusion systems has also been studied for different pressure laws and in the presence of drifts, see for instance \cite{KM18, DavidSchmidtchen2021, DEBIEC2020}.

\paragraph{Our contribution.}
In this paper, we aim to study the existence and regularity of solutions to System~\eqref{eq: prob pheno} and their incompressible limit. This problem can be seen as an infinitely-many-species extension of the models studied in \cite{price2020global,GPS, LX2021}. 
At first, we extend the method by \cite{price2020global} to the structured case. Adapting the same argument, we are able to prove the existence of global weak solutions, \cf Theorem~\ref{thm: existence pheno}.

The second main result of the paper, \cf Theorem~\ref{thm: limit problem} and Theorem~\ref{thm: complem relation}, concerns the  incompressible limit of System~\eqref{eq: prob pheno}. As $\gamma\to\infty$ in the pressure law, the problem turns out to be a free boundary problem of Hele-Shaw type. By extending and adapting the new method used in \cite{LX2021}, we are able to infer the compactness needed to pass to the limit. Moreover, by restricting our study to the class of compactly supported solutions, we are able to show strong compactness of the pressure $p_\gamma$ which, unlike in \cite{LX2021}, allows us to account for pressure-dependent reaction terms. 

Finally, we prove higher integrability results on the pressure gradient. In particular, we prove the uniform $L^4$-boundedness of $\nabla p_\gamma$, \cf Theorems~\ref{thm: L4 pheno}-\ref{remark: L4}, which has been introduced in the context of one-species porous medium models, see for instance \cite{MePeQu, DP, DavidSchmidtchen2021}, and represents a novelty in the multi-species case.

\paragraph{Plan of the paper.} In the next section, we present the assumptions and the main results of the paper. Section~\ref{sec: existence} is devoted to the proof of the existence of weak solutions: in Section~\ref{sec: preliminaries and reg} we introduce the regularised problem, obtained performing a viscosity perturbation, and we infer uniform a priori estimates, while in Section~\ref{sec: limit es to zero}, we show that $\nabla (\varrho_\es)^{\gamma+1}$ is strongly precompact in $L^2$, which is essential in order to pass to the limit in the regularised problem. In Section~\ref{sec: incompressible limit}, we study the asymptotics of Problem~\eqref{eq: prob pheno} as $\gamma\to\infty$. The additional regularity estimates are deduced in Section~\ref{sec: add estimates}.

\paragraph{Notation.} 
Given $T>0$ and $\Omega\subset\R^d$, we denote $Q_T:= \R^d\times(0,T), \Omega_T:= \Omega\times(0,T)$. We frequently use the abbreviated forms $n(t):=n(y,x,t),\, n(y):=n(y,x,t),\, \varrho(t):=\varrho(x,t)$.
Given a function $f$, we denote 
	\begin{equation*}
	\sign_+(f)=\mathds{1}_{\{f>0\}}\quad \text{ and } \quad \sign_-(f)=-\mathds{1}_{\{f<0\}}.
	\end{equation*}
	We also define the positive and negative part of $f$ as follows
	\begin{equation*}
	(f)_+ :=
	\begin{cases}
	f, &\text{ for } f>0,\\
	0, &\text{ for } f\leq 0,
	\end{cases}
	\quad  \text{ and }\quad
	(f)_- :=
	\begin{cases}
	-f, &\text{ for } f<0,\\
	0, &\text{ for } f\geq 0.
	\end{cases}
	\end{equation*}

\section{Assumptions and main results}

Now let us state the main results, \ie the existence of weak solutions to System~\eqref{eq: prob pheno}, the incompressible limit and the additional regularity estimates, and for each of them the related assumptions.

\subsection{Existence of weak solutions} 
\paragraph{Assumptions on the reaction term.}
The function $R(y,p)$ is assumed to be smooth and bounded. Moreover, since the pressure induces an inhibitory effect on cell proliferation, we suppose there exists a positive constant $p_M$ representing the \textit{homeostatic pressure}, such that
\begin{equation}\label{ass R_p}
 \partial_p R(\cdot,\cdot) \leq 0, \quad R(\cdot,0)> 0, \quad R(\cdot,p_M)\leq 0.
\end{equation}

\paragraph{Assumptions on the initial data.} 

Let us remind that the density fractions $\sigma(y,x):= n(y,x)/\varrho(x)$ are always well defined almost everywhere if we set $\sigma(y,x):=0$ where $\varrho(x)=0$. However, in order for the density fractions to be well defined everywhere and be always strictly positive, we regularize the initial data as follows $n_{0,\es}(y,x) = n_0(y,x) + \es e^{-|x|^2},$ \ie $\varrho_{0,\es}(x) = \varrho_0(x) + \es e^{-|x|^2},$ and  $p_{0,\es}=(\varrho_{0,\es})^\gamma.$

We say that the initial data are well-prepared if they satisfy the following assumptions: there exists $0<\es_0<1$ and $C$ independent of $\es$, such that for all $0<\es\leq \es_0$ the following holds
\begin{equation}\label{ass: well prepared id}
 0\leq \varrho_{0,\es}\leq (p_M)^{1/\gamma} \text{ a.e. in } \R^d, \qquad \left\|\sup_{y\in[0,1]}\frac{n_{0,\es}(y)}{\varrho_{0,\es}}\right\|_{L^\infty({\R^d})} \leq C,  \quad \left\|\varrho_{0,\es} |x|^2\right\|_{L^1(\R^d)}\leq C.
\end{equation}

\noindent
To show the existence of weak solutions, we extend the method developed in  \cite{price2020global} to the structured case and we prove the following result.

\begin{theorem}\label{thm: existence pheno}
Given $n_0\in L_+^\infty([0,1]\times \R^d)\cap L^1([0,1]\times \R^d)$ that satisfies Assumption~\eqref{ass: well prepared id}, there exists a weak solution to System~\eqref{eq: prob pheno}, namely, for all $T>0$ there exists $n(y,x,t)\in L_+^\infty([0,1]\times\R^d\times(0,T))\cap L^1([0,1]\times \R^d\times(0,T))$ such that $\nabla p(x,t)\in L^2(\R^d\times(0,T))$ and for all $\varphi \in C([0,1]; C_{c}^1([0,T)\times\R^d))$
 \begin{equation*}
 \begin{split}
      &-\int_0^1 \!\int_{\R^d} n(y,x,t)\partialt{\varphi(y,x,t)} \dx{x}\dx{y}+ \int_0^1\!\int_0^T\!\int_{\R^d} n(y,x,t) \nabla p(x,t)\cdot  \nabla\varphi(y,x,t)\dx{x}\dx{t}\dx{y} \\
      &\qquad
      = \int_0^1\!\int_0^T\!\int_{\R^d} n(y,x,t) R(y,p(x,t)) \varphi(y,x,t) \dx{x}\dx{t}\dx{y} + \int_0^1\!\int_{\R^d} n_0(y,x)\varphi(y,x,0) \dx{x}\dx{y}, 
 \end{split}
 \end{equation*}
with
\begin{equation*}
    \varrho(x,t)=\int_0^1 n(y,x,t) \dx{y},\;\; p(x,t)=\prt*{\varrho(x,t)}^\gamma, \;\; 0\leq p \leq p_M.
\end{equation*} 
\end{theorem}

\subsection{Incompressible limit}
In order to pass to the incompressible limit the more involved part is to find compactness of the pressure gradient. Our approach consists in extending and adapting the methods developed in \cite{LX2021} to our problem, namely focusing on the quantity $v_\gamma=\varrho_\gamma^{\gamma+1}$. 

Unlike \cite{LX2021}, we consider nonlinear pressure-dependent reaction terms. Consequently, our treatment of this term is different, and involves compensated compactness results and the monotonicity of $R$ with respect to $p$. Moreover, we need to assume that the solutions are compactly supported (uniformly in $\gamma$). Indeed, outside of this class of solutions we are not able to show the strong compactness of the pressure which is necessary in order to pass to the limit in the reaction terms. The problem then reduces to a boundary valued problem with Dirichlet homogeneous conditions, while in \cite{LX2021} the authors choose Neumann homogeneous conditions on the boundary. 

\bigskip

\noindent
\textbf{Assumptions on the initial data.}
We assume
$n_{\gamma,0}\in L^\infty([0,1]\times \R^d), \varrho_{\gamma,0}\in L_+^1(\R^d)\cap L^\infty(\R^d),$
and that there exists $\Omega_0\subset\R^d$ such that
$$ \supp\prt*{n_{\gamma,0}(y)}\subset \Omega_0, \ \text{for a.e. } y\in [0,1], \forall \gamma>1.$$
Thanks to the finite speed of propagation of porous medium type equations, we can reduce the problem to the case of a bounded domain $\Omega\subset\R^d$, on which we have homogeneous Dirichlet boundary conditions, $\varrho_\gamma(x,t)=0$, for almost every $(x,t)$ on $\partial\Omega \times [0,T]$. 
Since $\varrho_{\gamma,0}$ is compactly supported, then for all $T>0$ there exists $\Omega\subset\R^d$ such that
\begin{equation*}
    \supp{\varrho_\gamma(t)}\subset \Omega, \qquad \forall t\in[0,T], \ \forall \gamma>1.
\end{equation*} 
Moreover, we assume there exists $\varrho_0, p_0\in L_+^\infty(\Omega)$ such that
$$\|\varrho_{\gamma,0}-\varrho_0\|_{L^1(\Omega)}\to 0 \qquad \|p_{\gamma,0}-p_0\|_{L^1(\Omega)}\to 0 $$
and 
$$0\leq\varrho_{\gamma,0}\leq (p_M)^{1/\gamma}, \qquad 0\leq p_{\gamma,0}\leq p_M.$$

\bigskip
\noindent
Let us denote $v_\gamma = \varrho_\gamma^{\gamma+1}$. We can rewrite Eq.~\eqref{eq: density rho} as follows
\begin{equation}\label{eq: new density}
      \partialt{\varrho_\gamma} - \frac{\gamma}{\gamma+1}\Delta v_\gamma = \int_0^1 n_\gamma R(y,p_\gamma) \dx{y}.
\end{equation}
We can pass to the incompressible limit $\gamma \to \infty$ and recover a Hele-Shaw problem, as stated in the following theorems.
\bigskip

\begin{theorem}[Hele-Shaw weak limit]\label{thm: limit problem}
Let $(n_\gamma, \varrho_\gamma, p_\gamma)$ be a solution given by Theorem~\ref{thm: existence pheno} whose initial data satisfies the assumptions stated above. For all $T>0$, up to the extraction of a subsequence we have
  \begin{align} 
\label{convergence of n} n_\gamma(y,x,t) &\rightharpoonup n_\infty(y,x,t) & &\hspace{-3cm}  \text{ weakly$^*$ in } L^\infty((0,1)\times \Omega_T),\\[0.3em]
\label{convergence of varrho}\varrho_\gamma(x,t) &\rightharpoonup \varrho_\infty(x,t) & &\hspace{-3cm}\text{ weakly$^*$ in } L^\infty(\Omega_T),\\[0.3em]
\label{convergence of p}p_\gamma(x,t) &\rightharpoonup p_\infty(x,t) & &\hspace{-3cm} \text{ weakly$^*$ in } L^\infty(\Omega_T),\\[0.3em]
\label{convergence of nabla v} \nabla v_\gamma  &\rightharpoonup \nabla v_\infty & &\hspace{-3cm}\text{ weakly in } L^2(\OmegaT), 
\end{align}
as $\gamma\to\infty$. Moreover, the limit satisfies $v_\infty=p_\infty$, also stated as follows
\begin{equation}\label{saturation relation pheno}
    p_\infty(1-\varrho_\infty) = 0 \qquad \text{ almost everywhere in } \OmegaT,
\end{equation} as well as 
\begin{align}\label{limit weak eq for varrho}
    \partialt{\varrho_\infty} = \Delta p_\infty +  {\int_0^1 n_\infty R(y,p_\infty)\dx{y}},  \text{ in }\quad \curlyD'(\R^d \times (0,\infty)).
\end{align}
\end{theorem}
\noindent 
In order to pass to the limit in the equations for $n_\gamma$ and $p_\gamma$ we need to prove the strong compactness of $\nabla v_\gamma$ in $L^2(\OmegaT)$, see Lemma~\ref{lemma: strong convergence grad vgamma}. 
\begin{theorem}[Complementarity relation]\label{thm: complem relation}
The limit solution $\varrho_\infty, p_\infty$ satisfies 
\begin{equation} 
 \label{eq: p infty pheno}
    p_\infty \prt*{\Delta p_\infty + {\int_0^1 n_\infty R(y,p_\infty)\dx{y}}} = 0, \qquad \text{ in } \; \curlyD'(\R^d\times(0,\infty)).
\end{equation} 
as well as
\begin{equation} \label{limit weak equation for n(y)}
\begin{split}
-\int_0^1\iint_\OmegaT n_\infty &\partialt \varphi \dx{x} \dx{t} \dx{y}+  \int_0^1\iint_\OmegaT n_\infty \nabla p_\infty \cdot \nabla \varphi  \dx{x} \dx{t} \dx{y}\\[0.3em]
    &= \int_0^1\iint_\OmegaT n_\infty R(y,p_\infty)  \varphi\dx{x}\dx{t} \dx{y}  +\int_\Omega n_0 (y,x) \varphi(y,x,0)\dx{x}\dx{y}, 
\end{split}
\end{equation}
for all test functions $\varphi\in C((0,1); C_{comp}^1([0,T)\times\Omega))$.
\end{theorem}   

\noindent
Relation \eqref{saturation relation pheno} implies that the total limit density $\varrho_\infty$ is saturated in the positivity set of the pressure $\Omega(t):=\{x; \ p_\infty(x,t)>0\},$ which can be seen as the region occupied by the tumor.
Moreover, the \textit{complementarity relation} \eqref{eq: p infty pheno} tells us that in $\Omega(t)$ the limit pressure satisfies an elliptic equation, which is usually referred to as a Hele-Shaw free boundary problem.

\subsection{Additional regularity}
\label{subsec: regularity}
The last part of the paper concerns additional regularity estimates on the pressure gradient, therefore we focus on $p$ rather than $\varrho^{\gamma+1}$.  
In particular, we extend a result already proved in \cite{MePeQu} for a Hele-Shaw model of one species, which implies that $p^{\alpha-1} |\nabla p|^4$ is integrable, for certain values of $\alpha$. However, unlike \cite{MePeQu}, our bound does not rely on any uniform control of $\Delta p$, and allows us to prove a new bound in $L^4_{x,t}$ on the pressure gradient which is uniform in $\gamma$. The proof relies on controlling uniformly the $L^2$-norm of $p D^2 p$. Let us recall that $\Delta p$ is only a measure on the free boundary of $\{p >0\}$. Therefore, the control on the second derivatives is "weighted" by a factor $p$. This control has been recently applied to prove the existence of Lagrangian solutions for systems of porous medium equations \cite{jacobs2}. We also refer the reader to \cite{AB2020} for the derivation of a similar bound for some fluid dynamics models.

\paragraph{Additional assumptions.} In order to prove the following regularity results on the pressure, it is necessary to make stronger assumptions on the initial data. In particular, we assume 
 \begin{itemize}
    \setlength\itemsep{1em}
    \item[({\crtcrossreflabel{A1}[hyp:A1]})] $\gamma>3/2$,
    \item[({\crtcrossreflabel{A2}[hyp:A2]})] $\nabla p_{\gamma,0} \in L^2(\Omega)$ uniformly in $\gamma$.
\end{itemize}
\begin{theorem}\label{thm: L4 pheno}
Under assumptions (\ref{hyp:A1}-\ref{hyp:A2}), for all $T>0$ there exists a positive constant $C(T)$ independent of $\gamma$ such that for any $0\leq \alpha <\frac{1}{\gamma}$ the following estimate holds
\begin{equation*}
    \kappa(\alpha) \iint_{\Omega_T} \frac{|\nabla p|^4}{p^{1-\alpha}} \dx{x}\dx{t} \leq C(T),
\end{equation*}
with $\kappa(\alpha):= \frac{\alpha}{6} \prt*{1-\alpha \gamma}$.
\end{theorem}
\noindent
Exploiting the bounds obtained in the proof of the above theorem we also establish (uniformly in $\gamma$) an $L^4$-control of $\nabla p$.

 \begin{theorem}\label{remark: L4}
Under assumptions (\ref{hyp:A1}-\ref{hyp:A2}) the following estimates hold uniformly in $\gamma$,
\begin{equation*}
  \iint_{\Omega_T}\!p  (D^2_{i,j} p)^2\dx{x}\dx{t}\leq C(T), \qquad  \iint_{\Omega_T} |\nabla p|^4 \dx{x}\dx{t}
 \leq C(T).
 \end{equation*}
\end{theorem}
 
\section{Existence of weak solutions}\label{sec: existence}
\subsection{Regularised problem}\label{sec: preliminaries and reg}
In order to prove the existence of weak solutions of System~\eqref{eq: prob pheno}, we regularise the system introducing a viscosity term. Let $0<\es<\es_0$, and consider the following uniformly parabolic system
  \begin{equation}\label{eq: eq n_es}
    \begin{dcases} \partialt{n_\es} - \nabla \cdot(n_\es \nabla p_\es) - \es \Delta n_\es = n_\es R(y,p_\es), \qquad y \in [0,1],\quad (x,t)\in \Omega_T,\\
    \varrho_\es(x,t) = \int_0^1 n_\es(y,x,t) \dx{y}.
    \end{dcases}
 \end{equation}

\noindent
The equation on $\varrho_\es$ reads
\begin{equation}\label{eq: density rho_es}
   \partialt{\varrho_\es} - \frac{\gamma}{\gamma+1}\Delta \varrho_\es^{\gamma+1}- \es \Delta \varrho_\es = \int_0^1 n_\es R(y,p_\es)  \dx{y} .
\end{equation}
As mentioned above, in order to define the population fraction densities $\sigma_\es=n_\es/\varrho_\es$ we have to make sure that the total population density $\varrho_\es$ is always strictly positive. To this end, we regularise the initial data as follows
\begin{equation*}
    n_{0,\es}(y,x) = n_0(y,x) + \es \ e^{-|x|^2}, 
\end{equation*}
therefore
\begin{equation*}
    \varrho_{0,\es}(x) = \varrho_0(x) + \es \ e^{-|x|^2}.
\end{equation*}
For the existence and regularity of weak solutions of System~\eqref{eq: eq n_es} we refer the reader to \cite{price2020global}, where the authors use the same regularisation in order to show existence of weak solutions to the analogous two-species cross-diffusion system through a fixed point argument.

Before proving that the regularisation of the initial data implies strict positivity of $\varrho_\es(x,t)$ for all times, we prove non-negativity of solutions. 
\bigskip

\noindent
\textbf{Non-negativity.} Multiplying Eq.~\eqref{eq: eq n_es} by $\sign_-(n_\es)$ and using Kato's inequality we obtain
\begin{equation*}
     \partialt{}(n_\es)_- - \nabla \cdot((n_\es)_- \nabla p_\es) - \es \Delta (n_\es)_- \leq (n_\es)_- \|R\|_\infty, 
\end{equation*}
where we denote $\|R\|_\infty = \sup_{y\in[0,1]} R(y, 0)$. Integrating in space, we have
\begin{equation*}
     \ddt \int_{\R^d} (n_\es)_-\dx{x} - \int_{\R^d}\nabla \cdot((n_\es)_- \nabla p_\es) \dx{x}- \es\int_{\R^d} \Delta (n_\es)_- \dx{x}\leq\|R\|_\infty \int_{\R^d}(n_\es)_-\dx{x} , 
\end{equation*}
By Gronwall’s lemma we infer
\begin{equation*}
   \int_0^1\int_{\R^d} (n_\es(y,x,t))_-\dx{x}\dx{y} \leq e^{\|R\|_\infty t} \int_0^1\int_{\R^d} (n_{\es}(y,x,0))_- \dx{x}\dx{y},
\end{equation*}
which implies that almost everywhere $n_\es (t) \geq 0$
for $t\in (0, T]$ and by consequence both the density $\varrho_\es$ and the pressure $p_\es$ are non-negative.

\bigskip

\noindent{\textbf{Positivity.}}
Let us define the function
\begin{equation*}
    \underline{\varrho} = \es e^{-K t} e^{-|x|^2},
\end{equation*}
with $K= 2(\es + \gamma) + \|R\|_\infty$.
We state that $\underline{\varrho}$ is a subsolution of the following equation
\begin{equation*}
    \partialt \varrho = \frac{\gamma}{\gamma+1}\Delta \varrho^{\gamma+1} + \es \Delta \varrho -\varrho \|R\|_\infty.
\end{equation*}
In fact, we have
\begin{align*}
 \frac{\gamma}{\gamma+1}\Delta \underline\varrho^{\gamma+1} + \es \Delta \underline\varrho -\underline\varrho \|R\|_\infty = \ & 2\gamma \underline\varrho^{\gamma+1}(2(\gamma+1)|x|^2-1)+2\es (2|x|^2-1) \underline\varrho  -\underline\varrho\|R\|_\infty\\
 \geq &-2\es \underline\varrho - 2\gamma \underline\varrho^{\gamma+1} - \underline\varrho\|R\|_\infty\\
 \geq &(-2\es - 2\gamma - \|R\|_\infty)\underline\varrho\\
 =& - K \underline\varrho\\
 =& \partialt{\underline\varrho}.
\end{align*}
Therefore, since by \eqref{eq: density rho_es} $\varrho_\es$ is a supersolution to the same equation and $\varrho_\es(0) \geq \underline\varrho(0)$, by the comparison principle we have
\begin{equation*}
    \varrho_\es(t)\geq \underline\varrho(t) >0, \ 
\forall t\in[0,T].
\end{equation*}
Therefore, the quantity $$\sigma_\es(y,x,t):= \frac{n_\es(y,x,t)}{\varrho_\es(x,t)},$$ is well defined, and satisfies the following equation 
\begin{equation}\label{eq: fraction density}
    \partialt{\sigma_\es} = \varepsilon \Delta \sigma_\varepsilon + \frac{2 \varepsilon}{\varrho_\es} \nabla \sigma_\es \cdot\nabla \varrho_\es  + \nabla \sigma_\es \cdot \nabla p_\es + \sigma_\es R(y,p_\es) - \sigma_\es \int_0^1 \sigma_\es(\eta) R(\eta,p_\es)\dx{\eta},
 \end{equation}
where we used the notation $\eta$ to distinguish the variable of integration from the variable $y$ involved in the equation. 

Therefore, we rewrite the equation on $\varrho_\es$ as
 \begin{equation*}
   \partialt{\varrho_\es }-\frac{\gamma}{\gamma+1} \Delta \varrho_\es^{\gamma+1} - \es \Delta \varrho_\es = \varrho_\es \curlyR_\es,
 \end{equation*}
where we denote 
\begin{equation}\label{eq: curly R}
   \curlyR_\es:= \curlyR(\sigma_\es,p_\es) = \int_0^1 \sigma_\es(\eta) R(\eta,p_\es)\dx{\eta}.
\end{equation}
Let us notice that, from \eqref{eq: curly R}, $\curlyR_\es$ is also uniformly bounded in $L^\infty(Q_T)$ and 
\begin{equation*}
    \|\curlyR_\es\|_{L^\infty(Q_T)} \leq \sup_{y\in [0,1]}|R(y,0)| \int_0^1 \sigma_\es(\eta) \dx{\eta} = \|R\|_\infty.
\end{equation*}

\subsection{A priori estimates}\label{subsec: a priori}
Here we prove a priori estimates (uniform in $\es$) which are essential to prove the existence of weak solutions.
\bigskip

\noindent
\textbf{$\boldsymbol{L^1}$-bounds.} 
\noindent
Integrating in space we obtain
\begin{align*}
    \ddt \int_{\R^d} \varrho_\es \dx{x}&= \frac{\gamma}{\gamma+1}\int_{\R^d}\Delta \varrho_\es^{\gamma+1}\dx{x} +\es \int_{\R^d}\Delta \varrho_\es\dx{x} +  \int_{\R^d}\int_0^1 n_\es R(y,p_\es)\dx{y}\dx{x}\\[0.4em]
    &\leq \|R\|_\infty \int_{\R^d} \varrho_\es\dx{x}.
\end{align*}
By Gronwall's lemma we have $\varrho_\es\in L^\infty(0,T; L^1(\R^d))$ and thus $p_\es \in L^\infty(0,T; L^1(\R^d))$.

\bigskip
\noindent
\textbf{$\boldsymbol{L^\infty}$-bounds.}
\noindent
Let us denote $\varrho_M:= (p_M)^{1/\gamma}$. From Eq.~\eqref{eq: density rho_es} we have 
  \begin{align*}
     \partialt{}(\varrho_\es - \varrho_M) - \frac{\gamma}{\gamma+1}\Delta (\varrho_\es^{\gamma+1}-\varrho_M^{\gamma+1}) -\es \Delta(\varrho_\es -\varrho_M)   = \varrho_\es \curlyR_\es.
 \end{align*}
Multiplying by $\sign_+(\varrho_\es - \varrho_M)$ we obtain 
   \begin{equation*}
       \begin{split}
   \partialt{} (\varrho_\es - \varrho_M)_+ - \frac{\gamma}{\gamma+1}\Delta (\varrho_\es^{\gamma+1}-\varrho_M^{\gamma+1})_+ -\es \Delta(\varrho_\es -\varrho_M)_+ 
     \leq& \varrho_\es \curlyR_\es \sign_+(\varrho_\es - \varrho_M) \\[0.8em]
     \leq& 0,
       \end{split}
   \end{equation*}
where in the last inequality we used $\partial_p R \leq 0$ and $R(\cdot, p_M)\leq 0$.
Integrating over $\R^d$ and applying Gronwall's lemma we obtain
\begin{equation*}
    \ddt \int_{\R^d} (\varrho_\es - \varrho_M)_+ \dx{x} \leq e^{\|\curlyR\|_\infty t} \int_{\R^d} (\varrho_{0,\es}- \varrho_M)_+ \dx{x}.
\end{equation*}
For all $0<\es\leq\es_0$, thanks to Assumption~\eqref{ass: well prepared id}, we finally have
    \begin{equation}\label{eq: l infty rho p}
        0 \leq \varrho_\es \leq \varrho_M, \quad 0\leq p_\es\leq p_M.
    \end{equation}
Let us consider the equation on the fraction density, Eq.~\eqref{eq: fraction density}. By the assumptions on the reaction term, $\sigma_\es$ satisfies
\begin{equation*}
    \partialt{\sigma_\es} \leq \varepsilon \Delta \sigma_\varepsilon + \frac{2 \varepsilon}{\varrho_\varepsilon} \nabla \sigma_\varepsilon \cdot\nabla \varrho_\varepsilon + \nabla \sigma_\es \cdot \nabla p_\es + \sigma_\es 2 \|R_\es\|_\infty.
 \end{equation*}
Hence, by the comparison principle we obtain
    \begin{equation*}
        \sigma_\es\leq e^{2 \|R_\es\|_\infty t} \|\sigma_{0,\es}\|_\infty.    \end{equation*}
 Since by Assumption~\eqref{ass: well prepared id} $\sigma_{0,\es}$ is uniformly bounded in $[0,1] \times\R^d $, we have
 \begin{equation}\label{eq: l infty sigma}
  \sigma_\es \in L^\infty([0,1]\times Q_T)   , 
 \end{equation}
 and by consequence
 \begin{equation}\label{eq: l infty n}
      n_\es \in L^\infty([0,1]\times Q_T)  .
 \end{equation}

\bigskip
\noindent
\textbf{Bound on the pressure gradient.}
\noindent
Let us integrate the equation on the pressure $p_\es$, namely
\begin{equation*}
    \partialt{p_\es} = \gamma p_\es (\Delta p_\es + \curlyR_\es) +\es \Delta p_\es -\es \frac{\gamma-1}{\gamma} \frac{|\nabla p_\es|^2}{p_\es} +|\nabla p_\es|^2,
\end{equation*}
to obtain
\begin{equation*}
    \ddt \int_{\R^d} p_\es \dx{x} + (\gamma-1)\int_{\R^d} |\nabla p_\es|^2 \dx{x} \leq \gamma \int_{\R^d} p_\es \curlyR_\es \dx{x},
\end{equation*}
from which we infer $\nabla p_\es \in L^2(Q_T)$ uniformly in $\es$.

\bigskip
\noindent
\textbf{Second moment.}
\noindent
Now we prove the boundedness of the second moment of $\varrho_\es$.
Let us multiply Eq.~\eqref{eq: eq n_es} by $|x|^2/2$ to get
\begin{equation*}
    \partialt{n_\es} \frac{|x|^2}{2} = \nabla \cdot (n_\es \nabla p_\es) \frac{|x|^2}{2} + \es \Delta n_\es \frac{|x|^2}{2} + n_\es R(y,p_\es) \frac{|x|^2}{2}.
\end{equation*}
After integration by parts we have
\begin{equation*}
\begin{aligned}
    \ddt&\int_{\R^d}\int_0^1 n_\es \frac{|x|^2}{2}\dx{y}\dx{x} \\[0.5em]
    &= - \int_{\R^d}\int_0^1 n_\es \nabla p_\es \cdot x \dx{y}\dx{x} + \es d \int_{\R^d}\int_0^1  n_\es \dx{y}\dx{x}+ \int_{\R^d}\int^1_0 n_\es R(y,p_\es)\frac{|x|^2}{2}\dx{y}\dx{x}\\[0.5em]
    &\leq \int_{\R^d}\int_0^1 n_\es \frac{|x|^2}{2}\dx{y}\dx{x} + \frac{1}{2}\int_{\R^d}\int_0^1 n_\es |\nabla p_\es|^2 \dx{y}\dx{x} +  \es d \|\varrho_\es\|_{L^\infty(0,T;L^1(\R^d))}\\
    &\qquad\qquad \qquad + \|\curlyR_\es\|_\infty \int_{\R^d}\int_0^1 n_\es \frac{|x|^2}{2}\dx{y}\dx{x}\\[0.5em]
    &\leq C\int_{\R^d}\int_0^1 n_\es \frac{|x|^2}{2}\dx{y}\dx{x} +C.
\end{aligned}
\end{equation*}
By Gronwall's inequality we conclude that $n_\es |x|^2\in L^\infty(0,T; L^1([0,1]\times\R^d))$.

\subsection{Passing to the limit \texorpdfstring{$\boldsymbol{\es \to 0}$}{e-0}}
\label{sec: limit es to zero}
Extending the method by Price and Xu \cite{price2020global}, in this section, we prove the existence of solutions to System~\eqref{eq: prob pheno}, by showing the convergence of the solution of the regularised problem as $\es\to 0$. To this end, the most involved part consists in proving the strong convergence of the degenerate divergence term. Unlike the method developed by Gwiazda \textit{et al.} in \cite{GPS}, this strategy focuses on the quantity $\varrho_\es^{\gamma+1}$ rather than on the pressure $p_\es= \varrho_\es^\gamma$.
\begin{lemma}\label{lemma: entropy}
There exists a positive constant $C(T)$ independent of $\es$ such that the following holds
\begin{equation*}
 \frac{4\gamma}{(\gamma+1)^2} \iint_{Q_T} \left|\nabla \varrho_\es^{\frac{\gamma+1}{2}}\right|^2 \dx{x}\dx{t} + \es \iint_{Q_T} \int_0^1 \left|\nabla \sqrt{n_\es}(y)\right|^2 \dx{y} \dx{x}\dx{t} \leq C(T).
\end{equation*}
 \end{lemma}

\begin{proof}
Let $\nu$ be a positive constant. We multiply Eq.~\eqref{eq: eq n_es} by $\ln(n_\es + \nu)$ and we obtain
\begin{equation*}
    \partialt{n_\es} \ln(n_\es+\nu ) - \nabla\cdot(n_\es \nabla p_\es) \ln(n_\es + \nu) - \es \Delta n_\es \ln(n_\es + \nu) = n_\es R(y,p_\es) \ln(n_\es + \nu).
\end{equation*}
Integrating in space and in $y$ over $[0,1]$ we have
\begin{equation*}
\begin{aligned}
    \ddt \int_{\R^d}\!\int_0^1 \prt*{(n_\es + \nu) \ln(n_\es+\nu ) - n_\es} \dx{y}\dx{x} +& \int_{\R^d}\!\int_0^1 \frac{n_\es}{n_\es +\nu} \nabla p_\es \cdot \nabla n_\es \dx{y}\dx{x} + \es\int_{\R^d}\!\int_0^1 \frac{|\nabla n_\es|^2}{n_\es + \nu}  \dx{y}\dx{x}\\[0.4em]
    &= \int_{\R^d}\!\int_0^1n_\es R(y,p_\es) \ln(n_\es + \nu) \dx{y}\dx{x}\\[0.4em]
    &\leq \|R\|_\infty \int_{\R^d}\!\int_0^1 n_\es |\ln(n_\es+\nu)|  \dx{y}\dx{x}.
\end{aligned}
\end{equation*}
Let $t\leq T$. Upon integration in time for $\tau \in [0,t]$, we obtain
\begin{equation}\label{intermediate}
\begin{aligned}
\int_{\R^d}\!&\int_0^1 \prt*{(n_\es(t) + \nu) \ln(n_\es(t)+\nu )- n_\es(t)} \dx{y} \dx{x} +
 \int_0^t\!\int_{\R^d}\!\nabla p_\es \cdot \prt*{\int_0^1 \frac{n_\es}{n_\es +\nu}\nabla n_\es \dx{y}}\dx{x}\dx{\tau} \\[0.3em]
&\qquad + \es \int_0^t\!\int_{\R^d}\!\int_0^1 \frac{|\nabla n_\es|^2}{n_\es + \nu}  \dx{y}\dx{x}\dx{\tau} \\
&\leq  \int_{\R^d}\!\int_0^1 (n_{0,\es} + \nu) \ln(n_{0,\es}+\nu ) \dx{y}\dx{x} + \|R\|_\infty \int_0^t\int_{\R^d}\int_0^1 n_\es |\ln (n_\es + \nu)|\dx{y}\dx{x}\dx{\tau}.
\end{aligned}
\end{equation}
Let us now prove that $n_\es \ln n_\es$ is actually bounded in $L^\infty_t L^1_{y,x}$.
Using Fenchel-Young inequality we compute
\begin{equation*}
\begin{aligned}
    \int_{\R^d} n_\es |\ln n_\es| \dx{x} &= -\int_{\R^d} n_\es \ln n_\es \dx{x} +2\int_{\R^d} n_\es |\ln n_\es| \mathds{1}_{\{n_\es\geq 1\}}\dx{x}\\
    &\leq  -\int_{\R^d} n_\es \ln n_\es \dx{x} +2 \int_{\R^d} n_\es^2\mathds{1}_{\{n_\es\geq 1\}}\dx{x}\\
    &\leq \int_{\R^d} \big(e^{-|x|^2-1} + n_\es |x|^2 \big)\dx{x}  +2 \sup_{y\in[0,1]}\|n_\es(y)\|_{L^\infty(Q_T)}\int_{\R^d} n_\es\mathds{1}_{\{n_\es\geq 1\}}\dx{x}
    \end{aligned}
\end{equation*}
Finally, we can estimate
\begin{equation*}
    \begin{aligned}
      \sup_{0\leq t \leq T}  \int_{\R^d} \int_0^1 n_\es |\ln n_\es|\dx{y} \dx{x} &\leq  C ,
    \end{aligned}
\end{equation*}
where in the last inequality we used the uniform controls on the $L^1$-norm and the second moment of the total density.
Letting $\nu \to 0$ in \eqref{intermediate}, thanks the assumptions on $n_{0,\es}$ and the bound proven above, we have
\begin{equation}\label{entropy}
 \int_0^t\!\int_{\R^d}\!\nabla \varrho_\es^\gamma \cdot\nabla \varrho_\es \dx{x}\dx{\tau} + 4\es \int_0^t\!\int_{\R^d}\!\int_0^1 |\nabla \sqrt{n_\es}|^2 \dx{y}\dx{x}\dx{\tau} \leq  C(T),
\end{equation}
for all $0\leq t\leq T$, and this concludes the proof.

\end{proof}

\begin{remark}
From Eq.~\eqref{entropy} we also have (uniformly in $\es$ and $\gamma$)
    \begin{equation}\label{bound grad}
      \frac{1}{\gamma}  \int_0^T\!\int_{\R^d}  \frac{|\nabla p_\varepsilon|^2}{p_\varepsilon^{1-1/\gamma}}\dx{x}\dx{t} \leq C(T).
    \end{equation}
\end{remark}
\noindent
The above bound has been used in \cite{GPS} in order to prove a bound on $(\Delta p)_-$ in $L^\infty_tL^2_x$, namely an $L^p$-version of the celebrated Aronson-B\'enilan estimate. This requires a technical assumption on the reaction rates which in the proof of \cite[Theorem~1]{GPS} is combined to \eqref{bound grad}. However, this proof breaks down in the structured case, since the treatment of that particular term was strongly related to the fact that the number of species is finite. A way to remove that technical assumption was recently found in \cite{jacobs2}, where the author establishes a bound on a different quantity, however still sufficient to infer a bound of $\Delta p$ in $L^1_{x,t}$.

\begin{lemma}
The sequence $\varrho_\es^{\frac{\gamma+1}{2}}$ is precompact in $L^2(0,T; L^2({\R^d}))$.
\end{lemma}
\begin{proof}
From Lemma~\ref{lemma: entropy} we know that the gradient of $\varrho_\es^{\frac{\gamma+1}{2}}$ is bounded in $L^2(Q_T)$.
Now we compute its time derivative.
\begin{align*}
    \partialt{} \varrho_\es^{\frac{\gamma+1}{2}} = & \frac{\gamma+1}{2} \varrho_\es^{\frac{\gamma-1}{2}}\prt*{\nabla\cdot(\varrho_\es \nabla p_\es) + \es \Delta \varrho_\es + \int_0^1 n_\es(\eta)R(\eta,p_\es)\dx{\eta}}\\[0.3em]
    = &\frac{\gamma+1}{2} \varrho_\es^{\frac{\gamma-1}{2}} \nabla \cdot (\varrho_\es \nabla \varrho_\es^\gamma) + \frac{\gamma+1}{2}\es \varrho_\es^{\frac{\gamma-1}{2}} \Delta \varrho_\es+ \frac{\gamma+1}{2}\varrho_\es^{\frac{\gamma-1}{2}} \int_0^1 n_\es(\eta)R(\eta,p_\es) \dx{\eta}\\[0.3em]
    = &\frac{\gamma+1}{2} \nabla\cdot\prt*{\varrho_\es^{\frac{\gamma+1}{2}}\nabla\varrho_\es^\gamma}  - \frac{\gamma^2-1}{4} \varrho_\es^{\frac{\gamma-1}{2}} \nabla \varrho_\es \cdot\nabla\varrho_\es^{\gamma} + \frac{\gamma+1}{2}\es \nabla\cdot\prt*{\varrho_\es^{\frac{\gamma-1}{2}}\nabla \varrho_\es}\\[0.3em]
    &\quad- \frac{\gamma^2-1}{4}\es\varrho_\es^{\frac{\gamma-3}{2}}|\nabla\varrho_\es|^2 + \frac{\gamma+1}{2}\varrho_\es^{\frac{\gamma-1}{2}} \int_0^1 n_\es(\eta)R(\eta,p_\es) \dx{\eta}\\[0.3em]
    = &\gamma \nabla\cdot\prt*{\varrho_\es^\gamma \nabla \varrho_\es^{\frac{\gamma+1}{2}}} - \gamma\frac{\gamma-1}{\gamma+1}\varrho_\es^{\frac{\gamma-1}{2}}\left|\nabla \varrho_\es^{\frac{\gamma+1}{2}}\right|^2  +\es \Delta\varrho_\es^{\frac{\gamma+1}{2}} \\[0.3em]
    &\quad - \es (\gamma^2-1)\varrho_\es^{\frac{
    \gamma-1}{2}}|\nabla\sqrt{\varrho_\es}|^2 + \frac{\gamma+1}{2}\varrho_\es^{\frac{\gamma-1}{2}} \int_0^1 n_\es(\eta)R(\eta,p_\es) \dx{\eta}.
\end{align*}
Let us notice that Lemma~\ref{lemma: entropy} and the uniform $L^\infty$-bound of $\sigma_\es$ imply $\es|\nabla\sqrt{\varrho_\es}|^2 \in L^1(Q_T).$
Therefore, the time derivative of $ \varrho_\es^{\frac{\gamma+1}{2}}$ is a sum of functions bounded in $L^2(0,T; H^{-1}({\R^d}))$ and $L^1$-functions. Applying Aubin-Lions' lemma we infer that $ \varrho_\es^{\frac{\gamma+1}{2}}$ is precompact in $L^2(Q_T)$.

\end{proof}
\begin{remark}
The sequence $\varrho_\es$ is precompact in any $L^q$-space, for $1\leq q < \infty$.
In fact, if $q<\frac{\gamma+1}{2}$, the result follows from H\"older's inequality, while if $q>\frac{\gamma+1}{2}$ it follows from the uniform boundedness of $\varrho_\es$ in $L^\infty$.
\end{remark}

\begin{remark}\label{remark: convergence}
Let us recall the results already proven. Up to a subsequence, we have
\begin{align*}
    \sigma_\es &\rightharpoonup \sigma & &\hspace{-2.5cm}\text{weak}^* \text{ in } L^\infty([0,1]\times Q_T),
    \\[0.6em]
     n_\es &\rightharpoonup n   &  &\hspace{-2.5cm}\text{weak}^* \text{ in } L^\infty([0,1]\times Q_T),
    \\[0.6em]
     \varrho_\es &\rightarrow \varrho  & &\hspace{-2.5cm}\text{strongly} \text{ in } L^q(Q_T), \text{ for each } 1\leq q <\infty,
    \\[0.6em] 
     \varrho_\es^{\frac{\gamma+1}{2}} &\rightharpoonup  \varrho^{\frac{\gamma+1}{2}} & &\hspace{-2.5cm}\text{weakly} \text{ in } L^2(0,T; H^{1}({\R^d})),
    \\[0.6em]
    \partialt{\varrho_\es} &\rightharpoonup \partialt{\varrho} & &\hspace{-2.5cm}\text{weakly} \text{ in } L^2(0,T; H^{-1}({\R^d})).
\end{align*}
Let us recall the notation $\curlyR = \int_0^1 \sigma(\eta) R(\eta,p) \dx{\eta}$. Then
\begin{align}
\label{eq: conv curlyR}
    \curlyR_\es \rightharpoonup \curlyR \quad &\text{ weak$^*$ in } L^\infty(Q_T)\\[0.4em]
\label{eq: conv n R}
    n_\es R(y,p_\es) \rightharpoonup n R(y,p) \quad &\text{ weak$^*$ in } L^\infty([0,1]\times Q_T). 
\end{align}
\noindent
The convergences of \eqref{eq: conv curlyR} and \eqref{eq: conv n R} are shown in detail in Appendix~\ref{app: N}.
\end{remark}
 
\begin{lemma}\label{lemma: conv integrals}
For all $q\ge \gamma+1$ and all $t \in[0,T]$, we have
\begin{equation*}
    \int_{\R^d} \prt*{\varrho_\es(x,t)}^q \dx{x} \xrightarrow[]{\es\to 0} \int_{\R^d} \prt*{\varrho(x,t)}^q \dx{x}.
\end{equation*}
\end{lemma}

\begin{proof}
Let us define
\begin{equation*}
    w_\es := \varrho_\es^{\gamma+1} + \es \frac{\gamma+1}{\gamma} \varrho_\es.
\end{equation*}
Hence, we rewrite Eq.~\eqref{eq: density rho} as
\begin{equation}\label{eq: new equation density}
    \partialt{\varrho_\es} - \frac{\gamma}{\gamma +1} \Delta w_\es = \varrho_\es \curlyR_\es,
\end{equation}
where we recall that $\curlyR_\es=\int_0^1 \sigma_\es R(\eta,p_\es)\dx{\eta}$.
We test Eq.~\eqref{eq: new equation density} against $\partial_t w_\es$ to obtain
\begin{equation*}
\int_{\R^d} \partialt{\varrho_\es} \partialt{w_\es}\dx{x} - \frac{\gamma}{\gamma+1} \int_{\R^d} \Delta w_\es \partialt{w_\es} \dx{x}= \int_{\R^d} \varrho_\es \curlyR_\es \partialt{w_\es}\dx{x}.
\end{equation*}
Now we treat each term individually, to obtain
\begin{align*}
   \int_{\R^d} \partialt{\varrho_\es} \partialt{w_\es}\dx{x} =& \int_{\R^d} \partialt{\varrho_\es} \partialt{\varrho_\es^{\gamma+1}} \dx{x}+ \es \frac{\gamma+1}{\gamma}\int_{\R^d} \left|\partialt{\varrho_\es}\right|^2 \dx{x}
   \\[0.2em]
   =&(\gamma+1) \int_{\R^d} \varrho_\es^\gamma \left|\partialt{\varrho_\es}\right|^2 \dx{x}+ \es \frac{\gamma+1}{\gamma}\int_{\R^d} \left|\partialt{\varrho_\es}\right|^2\dx{x},\\[0.5em]
    -\frac{\gamma}{\gamma+1}\int_{\R^d} \Delta w_\es \partialt{w_\es}\dx{x} =& \frac{\gamma}{\gamma+1} \ddt \int_{\R^d} \frac{|\nabla w_\es|^2}{2}\dx{x},\\[0.5em]
   \int_{\R^d} \varrho_\es\curlyR_\es \partialt{w_\es}\dx{x} =& \int_{\R^d} \varrho_\es\curlyR_\es \partialt{\varrho_\es^{\gamma+1}}\dx{x} +\es \frac{\gamma+1}{\gamma} \int_{\R^d} \varrho_\es\curlyR_\es \partialt{\varrho_\es}\dx{x}\\[0.2em]
   \leq& \ \frac{\gamma+1}{2}\int_{\R^d} \varrho_\es^\gamma \left|\partialt{\varrho_\es}\right|^2\dx{x} + \frac{\gamma+1}{2}\int_{\R^d} \varrho_\es^{\gamma+2} \curlyR_\es^2\dx{x}\\[0.2em]
   &\quad+ \frac{\es}{2} \frac{\gamma+1}{\gamma} \int_{\R^d} \varrho_\es^2 \curlyR_\es^2\dx{x} + \frac{\es}{2} \frac{\gamma+1}{\gamma}\int_{\R^d} \left|\partialt{\varrho_\es}\right|^2\dx{x}.
\end{align*}
Therefore, we obtain
\begin{equation}\label{eq: 3 bounds}
    \sup_{t\in[0,T]} \int_{\R^d} |\nabla w_\es(t)|^2\dx{x} + \frac{\es}{2}\frac{\gamma+1}{\gamma}\iint_{Q_T} \left|\partialt{\varrho_\es}\right|^2\dx{x}\dx{t} + \frac{\gamma+1}{2}\iint_{Q_T} \varrho_\es^\gamma \left|\partialt{\varrho_\es}\right|^2 \dx{x}\dx{t}\leq C,
\end{equation}
where $C$ depends on $\|\varrho_\es\|_{\infty}$ and $\|\curlyR_\es\|_{\infty}$.
Since $\big|\partial_t \varrho_\es^{\frac{\gamma+2}{2}}\big|^2= \frac{(\gamma+2)^2}{4}\varrho_\es^\gamma |\partial_t \varrho_\es|^2$, from Eq.~\eqref{eq: 3 bounds} we have
\begin{align*}
    \partial_t \varrho_\es^{\frac{\gamma+2}{2}} \in L^2(Q_T), \qquad
    \sqrt{\es} \partial_t \varrho_\es \in L^2(Q_T), \qquad
    \nabla w_\es \in L^\infty(0,T; L^2({\R^d})).
\end{align*}
It follows easily from the boundedness of $\varrho_\es$, that $\partial_t\varrho_\es^{\gamma+1}\in L^2(Q_T).$ Hence, $\partial_t w_\es \in L^2(Q_T)$. Thanks to the bound on $\nabla w_\es$ and the Aubin-Lions lemma, $w_\es$ is precompact in $C([0,T], L^2({\R^d}))$. Consequently, $\varrho_\es^{\gamma+1}$ is also precompact in $C([0,T], L^2({\R^d}))$, since we have
\begin{equation*}
    \int_{\R^d} \left|\varrho_\es^{\gamma+1}(t) - \varrho^{\gamma+1}(t)\right|^2\dx{x} \leq  \int_{\R^d} \left|w_\es(t)
    - \varrho^{\gamma+1}(t)\right|^2 \dx{x}+ \int_{\R^d}\left|\es \frac{\gamma+1}{\gamma} \varrho_\es(t)\right|^2 \dx{x}\to 0, \: \text{as } \es \to 0.
\end{equation*}
Once again, thanks to the uniform boundedness of $\varrho_\es$ we infer that $\varrho_\es$ is precompact in $C([0,T], L^q({\R^d}))$ for any $q\geq \gamma+1.$ Therefore
\begin{equation*}
    \int_{\R^d} \prt*{\varrho_\es(x,t)}^q\dx{x} \xrightarrow[]{\es\to 0} \int_{\R^d} (\varrho(x,t))^q\dx{x}, \quad \forall q \geq \gamma+1,
\end{equation*}
and thus the proof is completed.
\end{proof}

\noindent
As already mentioned above, when dealing with cross-diffusion systems such as System~\eqref{eq: prob pheno}, the most involved part is to obtain the compactness needed to pass to the limit in the cross-diffusion term. In the absence of strong compactness of the single species densities, here being the distribution of each phenotypic trait $n_\es(y)$, it is essential to infer strong compactness of $\nabla \varrho_\es^{\gamma+1}$.
For this reason, the following convergence result is the core of the proof.
\begin{lemma}\label{thm: strong convergence grad}
Upon the extraction of a subsequence, we have
\begin{equation*}
    \nabla \varrho_\es^{\gamma+1} \xrightarrow[]{\es\to 0} \nabla \varrho^{\gamma+1} \quad \text{strongly in } L^2(Q_T).
\end{equation*} 
\end{lemma}
\begin{proof}
For the sake of simplicity, when integrating, we now neglect the symbols $\dx{x},\dx{t}$. Let us consider the limit equation
\begin{equation*}
    \partialt \varrho - \frac{\gamma}{\gamma+1}\Delta \varrho^{\gamma+1} = \varrho \curlyR,
\end{equation*}
and then subtract it from Eq.~\eqref{eq: density rho_es}, to obtain
\begin{equation*}
    \partialt{} (\varrho_\es - \varrho) + \frac{\gamma}{\gamma+1} \Delta (\varrho_\es^{\gamma+1}-\varrho^{\gamma+1}) + \es \Delta \varrho_\es = \varrho_\es \curlyR_\es - \varrho \curlyR.
    \end{equation*}
We test the above equation against $\varrho_\es^{\gamma+1}- \varrho^{\gamma+1}$ and we obtain
\begin{align*}
    \frac{\gamma}{\gamma+1}\iint_{Q_T} |\nabla(\varrho_\es^{\gamma+1}-\varrho^{\gamma+1})|^2 = &- \es \iint_{Q_T} \nabla \varrho_\es \cdot \nabla(\varrho_\es^{\gamma+1}-\varrho^{\gamma+1}) +\int_0^T \braket{\partial_t (\varrho_\es-\varrho),\varrho_\es^{\gamma+1}-\varrho^{\gamma+1}} \\[0.3em]
    &-\iint_{Q_T} (\varrho_\es \curlyR_\es - \varrho \curlyR)(\varrho_\es^{\gamma+1}-\varrho^{\gamma+1}).
\end{align*}
Let us consider the three terms on the right-hand side individually.
From to the strong convergence of $\varrho_\es$ in any $L^p$-space and the weak$^*$ convergence of $\curlyR_\es$, it directly follows that
\begin{equation*}
    \iint_{Q_T} (\varrho_\es \curlyR_\es - \varrho \curlyR)(\varrho_\es^{\gamma+1}-\varrho^{\gamma+1})\rightarrow 0.
\end{equation*}
Recalling Lemma~\ref{lemma: conv integrals}, the strong convergence of $\varrho_\es^{\gamma+1}$ and the weak convergence of $\partial_t \varrho_\es$ in $L^2(0,T; H^{-1}({\R^d}))$, we have
\begin{align*}
    \int_0^T \!\!\braket{\partial_t (\varrho_\es-\varrho),\varrho_\es^{\gamma+1}-\varrho^{\gamma+1}} = &\iint_{Q_T} \frac{\partial_t \varrho_\es^{\gamma+2}}{\gamma+2}+ \iint_{Q_T} \frac{\partial_t \varrho^{\gamma+2}}{\gamma+2}
    - \int_0^T\!\! \braket{\partial_t \varrho, \varrho_\es^{\gamma+1}}- \int_0^T\!\! \braket{\partial_t \varrho_\es,\varrho^{\gamma+1}}\\
    =&\int_{{\R^d}}\frac{\varrho_\es^{\gamma+2}(T)}{\gamma+2}+\int_{{\R^d}}\frac{\varrho^{\gamma+2}(T)}{\gamma+2} -\int_{{\R^d}}\frac{\varrho_\es^{\gamma+2}(0)}{\gamma+2}-\int_{{\R^d}}\frac{\varrho^{\gamma+2}(0)}{\gamma+2}\\
    &\quad - \int_0^T\!\! \braket{\partial_t \varrho, \varrho_\es^{\gamma+1}}- \int_0^T\!\! \braket{\partial_t \varrho_\es,\varrho^{\gamma+1}}\\
    \to& \ 2\int_{{\R^d}}\frac{\varrho^{\gamma+2}(T)}{\gamma+2} - 2\int_{{\R^d}}\frac{\varrho^{\gamma+2}(0)}{\gamma+2} -2 \int_0^T\!\! \braket{\partial_t \varrho,\varrho^{\gamma+1}} = 0.
\end{align*}
Since from Lemma~\ref{lemma: entropy} we know that $\sqrt{\es}\nabla\sqrt{\varrho_\es}$ and $\nabla\varrho_\es^{\frac{\gamma+1}{2}}$ are bounded in $L^2(Q_T)$, we finally compute
\begin{equation*}
    \es \iint_{Q_T} \nabla \varrho_\es \cdot \nabla(\varrho_\es^{\gamma+1}-\varrho^{\gamma+1})= 4\es \iint_{Q_T} \sqrt{\varrho_\es}\nabla\sqrt{\varrho_\es} \cdot \prt*{\varrho_\es^{\frac{\gamma+1}{2}}\nabla \varrho_\es^{\frac{\gamma+1}{2}} - \varrho^{\frac{\gamma+1}{2}} \nabla\varrho^{\frac{\gamma+1}{2}}}\leq \sqrt{\es} C \to 0,
\end{equation*}
and this concludes the proof.
\end{proof}
\noindent
Having proved the $L^2$-strong convergence of $\nabla \varrho_\es^{\gamma+1}$, we can now show that the limit of the sequence $(n_\es,\varrho_\es)$ is a solution of System~\eqref{eq: prob pheno}.

\bigskip 
\begin{proof}[Proof of Theorem~\ref{thm: existence pheno}]
For all $\varphi \in C([0,1]; C_{comp}^1([0,T)\times{\R^d}))$, the variational formulation of System~\eqref{eq: eq n_es} can be written as
 \begin{align} 
    \nonumber  - \int_0^1 &\int_{\R^d} n_\es(y,x,t)\partialt{\varphi(y,x,t)} \dx{x}\dx{y}+ \int_0^1\iint_{Q_T} n_\es(y,x,t) \nabla p_\es(x,t)\cdot  \nabla\varphi(y,x,t)\dx{x}\dx{t}\dx{y} \\[0.3em]
  =&- \es \int_0^1\!\iint_{Q_T} \nabla n_\es(y,x,t) \cdot  \nabla\varphi(y,x,t)\dx{x}\dx{t}\dx{y} \label{eq: prob pheno variational formulation epsilon}\\[0.3em]
  \nonumber    &+\int_0^1\iint_{Q_T}
n_\es(y,x,t) R(y, p_\es) \varphi(y,x,t) \dx{x}\dx{t}\dx{y} + \int_0^1\!\int_{\R^d} n_{0,\es}(y,x,t)\varphi(y,x,0) \dx{x}\dx{y}.
 \end{align}
As we already proved, there exists a bounded non-negative function $\sigma=\sigma(y,x,t)$ such that, up to a subsequence,
\begin{equation*}
    \sigma_\es\rightharpoonup \sigma \quad \text{ weakly}^* \text{ in } L^\infty([0,1]\times Q_T).
\end{equation*}
Therefore, from Lemma~\ref{thm: strong convergence grad} we infer
\begin{equation}\label{eq: conv nabla term}
\begin{split}
n_\es\nabla p_\es &= n_\es\nabla \varrho_\es^\gamma\\[0.3em]
    &= \sigma_\es\varrho_\es\nabla\varrho_\es^\gamma \\[0.3em]
    &= \sigma_\es\frac{\gamma}{\gamma+1} \nabla \varrho_\es^{\gamma+1} \xrightharpoonup[]{\es\to 0} \sigma \frac{\gamma}{\gamma+1}\nabla \varrho^{\gamma+1}, \quad \text{weakly in } L^2([0,1]\times Q_T).
    \end{split}
\end{equation}
Let us notice that $n(y,x,t)=\sigma(y,x,t)\varrho(x,t)$ almost everywhere, since we can pass to the limit $\varepsilon\to 0$ in the equation $n_\es(y,x,t)=\sigma_\es(y,x,t)\varrho_\es(x,t)$. 
Finally, using Eq.~\eqref{eq: conv nabla term}, Remark~\ref{remark: convergence} and passing to the limit in Eq.~\eqref{eq: prob pheno variational formulation epsilon} the proof is completed.
\end{proof}

\section{Incompressible limit}
\label{sec: incompressible limit}
Thanks to the result proven in the previous section, \cf Theorem~\ref{thm: existence pheno}, we know that for each $\gamma>1$ there exists $(n_\gamma, \varrho_\gamma, p_\gamma)$ that satisfies following equations
\begin{equation}\label{eq: weak form n gamma}
\begin{split}
      &-\int_0^1 \int_{\Omega} n_\gamma(y,x,t)\partialt{\varphi(y,x,t)} \dx{x}\dx{y} +\int_0^1\iint_{\Omega_T} n_\gamma(y,x,t) \nabla p_\gamma(x,t)\cdot  \nabla\varphi(y,x,t)\dx{x}\dx{t}\dx{y} \\
      &\qquad\qquad\;
      =\int_0^1\iint_{\Omega_T} n_\gamma(y,x,t) R(y, p_\gamma) \varphi(y,x,t) \dx{x}\dx{t}\dx{y} + \int_0^1\int_{\Omega} n_{\gamma,0}(y,x,t)\varphi(y,x,0) \dx{x}\dx{y}, 
\end{split}
\end{equation}
 for all $\varphi \in C([0,1]; C_{comp}^1([0,T)\times\Omega))$
\begin{equation}
\label{eq: weak form varrho}  
\begin{split}
  -\iint_\OmegaT& \varrho_\gamma(x,t) \partialt \psi(x,t) \dx{x} \dx{t}+ \frac{\gamma}{\gamma+1}\iint_\OmegaT \nabla v_\gamma(x,t) \cdot \nabla \psi(x,t)  \dx{x} \dx{t} =\\[0.6em] &\iint_\OmegaT\prt*{\int_0^1 n_\gamma(x,t) R(y,p_\gamma(x,t)) \dx{y}}  \psi(x,t)\dx{x}\dx{t} +\int_\Omega \varrho_{\gamma,0} (x) \psi(x,0)\dx{x},
    \end{split}
\end{equation}
for all test functions $\psi\in C_{comp}^1([0,T)\times\Omega)$, where $v_\gamma = \varrho_\gamma^{\gamma+1}$.

\bigskip
\noindent
The goal of this section is to study the incompressible limit $\gamma\to\infty$ and recover the weak formulation of a Hele-Shaw free boundary problem.
To this end, we have to infer the compactness on the main quantities needed to pass to the limit in (\ref{eq: weak form n gamma}, \ref{eq: weak form varrho}). While for the first equation, the strong compactness of $\nabla p_\gamma$ is needed, weak compactness of $\nabla v_\gamma$ is sufficient in order to pass to the limit in Eq.~\eqref{eq: weak form varrho}, as stated in Theorem~\ref{thm: limit problem}.

The second main result is the \textit{complementarity relation},  Theorem~\ref{thm: complem relation}, which allows to establish the limit pressure as the solution of an elliptic equation. In order to prove it we need to infer the strong compactness of $\nabla p_\gamma$, which also allows us to pass to the limit in Eq.~\eqref{eq: weak form n gamma}. 

The following part of this section is devoted to the proof of Theorem~\ref{thm: limit problem} and Theorem~\ref{thm: complem relation}. Since we are not able to prove any control on $\partial_t p_\gamma$, it is not possible to directly prove the strong compactness of $p_\gamma$ (Corollary~\ref{coro: strong c p}) which is necessary in order to find the limit of the reaction term. For this reason we will be able to identify the limit only after the proof of the strong compactness of $\nabla v_\gamma$ (Lemma~\ref{lemma: strong convergence grad vgamma}).

\bigskip

\subsection{Proof of Theorem~\ref{thm: limit problem}}

\begin{remark}[Weak$^*$ convergence as $\gamma\to \infty$]
Let us point out that the $L^\infty$-bounds \eqref{eq: l infty rho p},\eqref{eq: l infty sigma} and \eqref{eq: l infty n} proven in Subsection~\ref{subsec: a priori} are also uniform with respect to $\gamma$. Therefore, there exist $n_\infty, \varrho_\infty, p_\infty$ and $v_\infty$ such that, after the extraction of a subsequence Eqs.~\eqref{convergence of n}-\eqref{convergence of p} hold. Moreover, there exists $\curlyH_\infty$ such that 
\begin{equation}\label{conv curly H}
n_\gamma R(y,p_\gamma) \rightharpoonup \curlyH_\infty  \text{ weakly$^*$ in } L^\infty((0,1)\times \Omega_T).
\end{equation}
\end{remark}

\begin{remark}[$H^1$-bounds of $p_\gamma$ and $v_\gamma$]
Multiplying the equation on the density, Eq.~\eqref{eq: density rho}, by $\gamma \varrho_\gamma^{\gamma-1}$, it is immediate to see that the pressure satisfies
\begin{equation}\label{eq: pressure pheno}
    \partialt {p_\gamma}= \gamma p_\gamma (\Delta p_\gamma + \curlyR_\gamma)+ |\nabla p_\gamma|^2.
\end{equation}
Hence, the pressure gradient is bounded in $L^2(\Omega_T)$ as shown by integrating by parts in space to get
\begin{equation*}
    \ddt\int_{\Omega} p_\gamma \dx{x} =(1- \gamma) \int_\Omega |\nabla p_\gamma|^2  \dx{x}+\gamma \int_\Omega p_\gamma \curlyR_\gamma \dx{x},
\end{equation*}
which implies
\begin{equation*}
    (\gamma -1) \iint_{\Omega_T} |\nabla p_\gamma|^2  \dx{x} \dx{t}\leq \gamma \|\curlyR_\gamma\|_{L^\infty(\Omega_T)} \|p_\gamma\|_{L^1(\Omega_T)} +\|p_0\|_{L^1(\Omega)}.
\end{equation*}
  Therefore, for all $\gamma>1$, it holds
  \begin{equation}
    \label{eq: grad p in L2}
    p_\gamma \in L^2(0,T;H^1(\Omega)).
  \end{equation}
By the definition of $v_\gamma$, we have
\begin{equation}\label{eq L2 bound grad v}
    \nabla v_\gamma = \frac{\gamma+1}{\gamma} p_\gamma^{\frac{1}{\gamma}} \nabla p_\gamma =  \frac{\gamma+1}{\gamma} \varrho_\gamma \nabla p_\gamma  \in L^2(\Omega_T),
\end{equation}
uniformly in  $\gamma$, and therefore Eq.~\eqref{convergence of nabla v} is proven.  
\end{remark}

\begin{corollary}
The limit triplet $(n_\infty, \varrho_\infty, p_\infty)$ satisfies
\begin{equation}
\label{eq: limit varrho}
       \partialt{\varrho_\infty} = \Delta v_\infty + \int_0^1\curlyH_\infty(y) \dx{y}, \text{ in } \;\; \curlyD'(\R^d\times(0,\infty)),
\end{equation}
where $\curlyH_\infty=\curlyH_\infty(y,x,t)$ is the weak limit of $n_\gamma R(y,p_\gamma)$.
\end{corollary}
\begin{proof}
The result comes from passing to the limit in Eq.~\eqref{eq: weak form varrho} using the convergence results~\eqref{convergence of varrho}, \eqref{convergence of nabla v}, and \eqref{conv curly H}.
\end{proof}

\noindent
As mentioned above, in order to conclude the proof of \eqref{limit weak eq for varrho} we have to show that $\curlyH_\infty= n_\infty R(y,p_\infty)$. This will be proven in the following subsection, \cf Eq.~\eqref{eq: H}. At this moment, we are not able to identify the limit since we do not have the strong compactness of $p_\gamma$.

\begin{remark}[$H^{-1}$-bound of the density time-derivative]
From the previous bounds and Eq.~\eqref{eq: new density}, we have 
\begin{equation}\label{eq: dt rho H-1}
    \partialt{\varrho_\gamma} \in L^2(0,T; H^{-1}(\Omega)). 
\end{equation}
\end{remark}

\begin{corollary}
The limit solution satisfies Eq.~\eqref{saturation relation pheno}. 
\end{corollary}
\begin{proof}
Let us recall that the non-negativity of $n_\gamma$, and consequently of $\varrho_\gamma$ and $p_\gamma$, has already been proven in the previous sections. Since $\varrho_\gamma \leq \varrho_M = (p_M)^{1/\gamma}$ we have $0\leq \varrho_\infty \leq 1$. 

By definition we have $v_\gamma= \varrho_\gamma p_\gamma$. Thanks to Eqs.~\eqref{eq: grad p in L2} and \eqref{eq: dt rho H-1} we can apply the compensated compactness theorem stated in Appendix~\ref{app: cc}, \cf Theorem \ref{thm: comp comp}, and infer
\begin{equation*}
  \int_\OmegaT  v_\gamma \varphi \dx{x}\dx{t}\rightarrow \int_\OmegaT\varrho_\infty p_\infty \varphi\dx{x}\dx{t},
\end{equation*}
for every $\varphi\in C(0,T; C^1(\Omega))$.
Hence $v_\infty= \varrho_\infty p_\infty,$ almost everywhere.
Finally, by weak lower semi-continuity of convex functions we have
\begin{equation*}
   \lim_{\gamma\to\infty} v_\gamma=   \liminf_{\gamma\to\infty} p_\gamma^{\frac{\gamma+1}{\gamma}}\geq p_\infty.
\end{equation*}
For the sake of completeness, we include here the full argument. Let us denote $\Psi_\gamma(x)= x^{\frac{\gamma+1}{\gamma}}$, $\gamma>1.$ Let $\psi_\delta=\psi_\delta(x)$ be a convex function such that $\psi_\delta(x)\to x$ as $\delta\to 0,$ and
\begin{equation*}
    \psi_\delta (x)\leq \Psi_\gamma(x), \qquad \text{for } \gamma \text{ large enough}.
\end{equation*}
For example, we could take
\begin{equation*}
\psi_\delta(x):=\begin{cases}
    0, &\text{for } 0\leq x \leq \delta,\\
    x-\delta &\text{for } x>\delta.
    \end{cases}
\end{equation*}
Therefore, we have
\begin{equation*}
   \psi_\delta (p_\infty)\leq \liminf_{\gamma\to\infty}\psi_\delta (p_\gamma)\leq \liminf_{\gamma\to\infty}\Psi_\gamma(p_\gamma)=\liminf_{\gamma\to\infty}p_\gamma^{\frac{\gamma+1}{\gamma}}. 
\end{equation*}
Since we chose $\delta>0$ arbitrarily, we take $\delta\to 0$ to obtain 
\begin{equation*}
   p_\infty\leq\liminf_{\gamma\to\infty}p_\gamma^{\frac{\gamma+1}{\gamma}}. 
\end{equation*}
Hence $\varrho_\infty p_\infty= v_\infty \geq p_\infty$, which implies $\varrho_\infty p_\infty = p_\infty$.

\end{proof}

\subsection{Proof of Theorem~\ref{thm: complem relation}}
In order to prove the complementarity relation, \cf Theorem~\ref{thm: complem relation}, one possible strategy is to infer the strong convergence of $\nabla p_\gamma$ through regularity, see for instance \cite{DP, DavidSchmidtchen2021, BPPS}. Since those methods cannot be applied to our case, we follow the strategy of \cite{LX2021}, directly proving the strong compactness of $\nabla v_\gamma$.
The core of the proof is given by the following lemma.
\begin{lemma}\label{lemma: strong convergence grad vgamma} Up to a subsequence, as $\gamma\to\infty$, we have
\begin{align}
  \label{eq: compactness of grad v}   \nabla v_\gamma \rightarrow \nabla v_\infty \quad &\text{ strongly in } L^2(\Omega_T).
\end{align}
\end{lemma}
\begin{proof}
Let us use $v_\gamma - v_\infty$ as a test function in Eq.~\eqref{eq: new density} to obtain
\begin{equation}
    \label{eq: step 1}
    \begin{split}
         \int_\Omega\partialt{\varrho_\gamma} (v_\gamma-v_\infty) \dx{x}+ \frac{\gamma}{\gamma+1}\int_\Omega\nabla v_\gamma \cdot \nabla (v_\gamma - v_\infty)\dx{x}= \int_\Omega\prt*{\int_0^1 n_\gamma R(y,p_\gamma) \dx{y}} (v_\gamma- v_\infty) \dx{x}.
    \end{split}
\end{equation}
We note that
\begin{equation*}
    \int_\Omega\partialt{\varrho_\gamma} v_\gamma \dx{x} = \frac{1}{\gamma+2}\int_\Omega \partialt{\varrho_\gamma^{\gamma+2}} \dx{x}=   \frac{1}{\gamma+2}\ddt\int_\Omega {\varrho_\gamma^{\gamma+2}} \dx{x}.
\end{equation*}
Integrating in time we get
\begin{equation*}
    \iint_{\Omega_T}\partialt{\varrho_\gamma} v_\gamma \dx{x}\dx{t} = \frac{1}{\gamma+2}\int_\Omega {\varrho_\gamma^{\gamma+2}(T)} \dx{x}- \frac{1}{\gamma+2}\int_\Omega {\varrho_\gamma^{\gamma+2}(0)} \dx{x} \to 0,
\end{equation*}
as $\gamma \to \infty$.
Now we compute
\begin{equation}\label{eq: step1.5}
\begin{split}
    \limsup_{\gamma\to\infty} \iint_\OmegaT &|\nabla(v_\gamma- v_\infty)|^2 \dx{x}\dx{t}\\
    \leq& \limsup_{\gamma\to\infty} \prt*{\iint_\OmegaT \nabla v_\gamma \cdot \nabla (v_\gamma -v_\infty)\dx{x}\dx{t} - \iint_\OmegaT \nabla v_\infty \cdot \nabla(v_\gamma-\nabla v_\infty)\dx{x}\dx{t}}\\
    \leq&\limsup_{\gamma\to\infty} \iint_\OmegaT \nabla v_\gamma \cdot \nabla (v_\gamma -v_\infty)\dx{x}\dx{t},
    \end{split}
\end{equation}
where in the last inequality we use the fact that $\nabla v_\gamma$ is weakly compact in $L^2(\OmegaT)$.
From Eq.~\eqref{eq: step 1} we obtain
\begin{equation}\label{eq: step 2}
\begin{split}
  \limsup_{\gamma\to\infty}& \iint_\OmegaT \nabla v_\gamma \cdot \nabla (v_\gamma -v_\infty)\dx{x}\dx{t}\\
  \leq &\limsup_{\gamma\to\infty} \iint_\OmegaT \prt*{\int_0^1 n_\gamma R(y,p_\gamma) \dx{y}} (v_\gamma- v_\infty)\dx{x}\dx{t}+ \limsup_{\gamma\to\infty} \iint_\OmegaT \partialt{\varrho_\gamma} v_\infty \dx{x}\dx{t}\\
  \leq &\limsup_{\gamma\to\infty} \iint_\OmegaT \prt*{\int_0^1 n_\gamma R(y,p_\gamma) \dx{y}} (v_\gamma- v_\infty)\dx{x}\dx{t}+\iint_\OmegaT \partialt{\varrho_\infty} v_\infty \dx{x}\dx{t},
  \end{split}
\end{equation}
where we used the  weak compactness of the density in $L^2(0,T; H^{-1}(\Omega))$ given by Eq.~\eqref{eq: dt rho H-1}. We now treat the first term in the right-hand side of Eq.~\eqref{eq: step 2}.
We add and subtract the same quantity to get
\begin{equation*}
\begin{split}
    \iint_\OmegaT \prt*{\int_0^1 n_\gamma R(y,p_\gamma) \dx{y}} (v_\gamma- v_\infty)\dx{x}\dx{t}=&\underbrace{\iint_\OmegaT \prt*{\int_0^1 n_\gamma (R(y,p_\gamma)-R(y,p_\infty)) \dx{y}} (v_\gamma- v_\infty)\dx{x}\dx{t}}_{\curlyA}\\[0.3em]
    &\qquad + \underbrace{\iint_\OmegaT \prt*{\int_0^1 n_\gamma R(y,p_\infty) \dx{y}} (v_\gamma- v_\infty)\dx{x}\dx{t}}_{\curlyB}.
    \end{split}
\end{equation*}
Our goal is to prove that the right hand side is bounded by some quantity that converges to zero as $\gamma\to\infty$. To deal with $\curlyA$ we use the monotonicity of $R(y,\cdot)$, which is a decreasing function of the pressure.
We rewrite $\curlyA$ as follows
\begin{equation*}
    \begin{split}
 \curlyA&=\iint_\OmegaT \prt*{\int_0^1 n_\gamma (R(y,p_\gamma)-R(y,p_\infty)) \dx{y}} (p_\gamma\varrho_\gamma -v_\infty)\dx{x}\dx{t}   \\
 &=\iint_\OmegaT \prt*{\int_0^1 n_\gamma (R(y,p_\gamma)-R(y,p_\infty))\dx{y}} (p_\gamma(\varrho_\gamma-1) +p_\gamma-p_\infty)\dx{x}\dx{t}  \\
 &=\iint_\OmegaT \prt*{\int_0^1 n_\gamma (R(y,p_\gamma)-R(y,p_\infty))\dx{y}} p_\gamma(\varrho_\gamma-1) \dx{x}\dx{t}  \\
 &\qquad+\iint_\OmegaT \prt*{\int_0^1 n_\gamma (R(y,p_\gamma)-R(y,p_\infty))\dx{y}}( p_\gamma-p_\infty)  \dx{x}\dx{t},
    \end{split}
\end{equation*}
where the last integral is non-positive by the monotonicity of $R$.
Let $\es>0$, we split the remaining term as follows
\begin{equation*}
    \begin{split}
        &\iint_\OmegaT \prt*{\int_0^1 n_\gamma (R(y,p_\gamma)-R(y,p_\infty)) \dx{y}} p_\gamma(\varrho_\gamma-1) \dx{x}\dx{t} \\[0.7em]
        = &\iint_{\OmegaT\cap\{\varrho_\gamma\leq 1-\es\}} \prt*{\int_0^1 n_\gamma (R(y,p_\gamma)-R(y,p_\infty)) \dx{y}} \varrho_\gamma^\gamma(\varrho_\gamma-1) \dx{x}\dx{t}\\[0.7em]
        &\qquad+\iint_{\OmegaT\cap\{\varrho_\gamma>1-\es\}} \prt*{\int_0^1 n_\gamma (R(y,p_\gamma)-R(y,p_\infty)) \dx{y}} p_\gamma(\varrho_\gamma-1) \dx{x}\dx{t}\\[0.7em]
         \leq & \ 2\|R\|_\infty \varrho_M (1-\es)^\gamma+ 2\|R\|_\infty \varrho_M p_M \max\prt*{\es, \frac 1 \gamma |\ln p_M| + o\prt*{\frac 1 \gamma}}.
    \end{split}
\end{equation*}
Choosing $\es=1/\sqrt{\gamma}$, we infer that the right-hand side converges to zero as $\gamma\to\infty$. 

Now we show that, after the extraction of a subsequence, the term 
$$\curlyB=\int_0^1  \prt*{\iint_\OmegaT n_\gamma R(y,p_\infty)(v_\gamma- v_\infty)\dx{x}\dx{t} } \dx{y},$$ 
converges to zero as $\gamma \to \infty$.
Let us choose $y\in(0,1)$. We denote $w_{\gamma}:= R(y,p_\infty)(v_\gamma- v_\infty)$. First of all, there exists a subsequence $\gamma_k$ independent of $y$ such that $w_{\gamma_k}$ converges to zero weakly in $L^2(\Omega_T)$. Let us recall that
\begin{equation*}
       \partial_t n_\gamma(y) = \nabla \cdot(n_\gamma(y) \nabla p_\gamma) +  n_\gamma(y) R(y,p_\gamma).
\end{equation*}
Hence, $\partial_t n_\gamma(y)\in L^2(0,T; H^{-1}(\Omega))$. Therefore, we can apply the compensated compactness theorem, see Theorem~\ref{thm: comp comp}. For all indexes $\gamma_{k_j}$ there exist $\gamma_{k_{j_i}}$ such that
$$\iint_\OmegaT n_{\gamma_{k_{j_i}}}(y) R(y,p_\infty)(v_{\gamma_{k_{j_i}}}- v_\infty)\dx{x}\dx{t}\to 0,$$
as $i\to \infty,$ which implies
$$\iint_\OmegaT n_{\gamma_{k}}(y) R(y,p_\infty)(v_{\gamma_{k}}- v_\infty)\dx{x}\dx{t}\to 0,$$
as $k\to \infty.$ Moreover, the above function is uniformly bounded in $L^1([0,1]).$
Since $\gamma_k$ only depends on the convergence of $v_\gamma$ we have
$$\curlyB=\int_0^1  \prt*{\iint_\OmegaT n_{\gamma_k} R(y,p_\infty)(v_{\gamma_k}- v_\infty)\dx{x}\dx{t} } \dx{y}\to 0,$$
as $k\to\infty$.

Now, we can finally come back to Eqs.\eqref{eq: step1.5}-\eqref{eq: step 2}
\begin{equation} \label{eq: important inequality}
\begin{split}
   \limsup_{\gamma\to\infty}\iint_\OmegaT |\nabla (v_\gamma -v_\infty)|^2\dx{x}\dx{t}\leq  \iint_\OmegaT \partialt{\varrho_\infty} v_\infty \dx{x}\dx{t}.
  \end{split}
\end{equation}
To conclude the proof we will show that the right-hand side is actually equal to zero.
Let us notice that for any $\es>0$
\begin{equation*} 
\begin{split} 
\iint_\OmegaT (\varrho_\infty(x, t+\es) - \varrho_\infty(x,t)) v_\infty \dx{x}\dx{t} =\iint_\OmegaT (\varrho_\infty(x, t+\es)-1+1 - \varrho_\infty(x,t)) v_\infty \dx{x}\dx{t} \leq 0,
  \end{split}
\end{equation*}
where in the last inequality we used Eq.~\eqref{saturation relation pheno}. In a similar fashion we have
\begin{equation*} 
\begin{split} 
\iint_\OmegaT (\varrho_\infty(x, t) - \varrho_\infty(x,t-\es)) v_\infty \dx{x}\dx{t} \geq 0.
  \end{split}
\end{equation*}
Now it remains to prove that 
\begin{equation}
    \label{eq: es to 0}
    \lim_{\es \to 0} \iint_\OmegaT  (\varrho_\infty(x, t+\es) - \varrho_\infty(x,t)) v_\infty  \dx{x}\dx{t} =  \iint_\OmegaT  \partialt{\varrho_\infty} v_\infty \dx{x}\dx{t}.
\end{equation}
We integrate Eq.~\eqref{eq: limit varrho} between $t$ and $t+\es$ to obtain
\begin{equation*}
 \varrho_\infty(t+\es) - \varrho_\infty(t) =\int_t^{t+\es} \Delta v_\infty \dx{s}+ \int_t^{t+\es}\int_0^1\curlyH_\infty\dx{y} \dx{s}.
\end{equation*}
We test the above equation against $\frac 1 \es v_\infty(\cdot, t)$ to get
\begin{equation}\label{eq: t+es - t}
\begin{split}
 \int_\Omega \prt*{\frac{\varrho_\infty(x,t+\es) - \varrho_\infty(x,t)}{\es}} v_\infty(x,t) \dx{x}=-\int_\Omega \frac 1 \es\int_t^{t+\es} \nabla v_\infty(x,s) \dx{s} \ \cdot \nabla v_\infty(x,t) \dx{x}\\
 + \int_\Omega \frac 1 \es \int_t^{t+\es}\int_0^1\curlyH_\infty(y,x,s)\dx{y}  \dx{s} \ v_\infty(x,t)\dx{x}.
\end{split}
\end{equation}
We have
\begin{equation*}
  \frac 1 \es\int_t^{t+\es} \nabla v_\infty(x,s) \dx{s} \rightarrow \nabla v_\infty (x,t), \quad \text{ a.e. in }   \OmegaT.
\end{equation*}
From Eq.~\eqref{eq L2 bound grad v} we have
\begin{equation*}
    \begin{split}
    \iint_\OmegaT \left| \frac 1 \es\int_t^{t+\es} \nabla v_\infty(x,s) \dx{s}\right|^2 \dx{x}\dx{t} \leq& \frac 1 \es \iint_\OmegaT \int_t^{t+\es} |\nabla v_\infty(x,s)|^2 \dx{s} \dx{x}\dx{t}\\
        =&\frac 1 \es \int_0^{T+\es} \int_{\max(0, s-\es)}^{\min(T,s)}\int_\Omega |\nabla v_\infty(x,s)|^2 \dx{x}\dx{t}\dx{s}\\
        \leq& \frac{1}{\es} \int_0^{T+\es}  |\min(T,s)-\max(0,s-\es)| \int_\Omega |\nabla v_\infty(x,s)|^2 \dx{x} \dx{s} \\
        \leq& \ C(T).
    \end{split}
\end{equation*}
Therefore we have
\begin{equation*}
   \frac 1 \es\int_t^{t+\es} \nabla v_\infty(x,s) \dx{s} \rightarrow \nabla v_\infty (x,t), \quad \text{weakly in } L^2(\OmegaT).
\end{equation*}
In an analogous way we can prove that
\begin{equation*}
   \frac 1 \es\int_t^{t+\es} \int_0^1\curlyH_\infty(y,x,s)\dx{y}  \dx{s} \rightarrow \int_0^1\curlyH_\infty(y,x,t)\dx{y}, \quad \text{weakly in } L^2(\OmegaT).
\end{equation*}
Combining Eq.~\eqref{eq: t+es - t} and Eq.~\eqref{eq: limit varrho} we have
\begin{equation*} 
\begin{split}
&\lim_{\es\to 0} \iint_\OmegaT \prt*{\frac{\varrho_\infty(t+\es) - \varrho(t)}{\es}} v_\infty(x,t) \dx{x}\dx{t}\\[0.3em]
&\qquad=-\iint_\OmegaT |\nabla v_\infty|^2 \dx{x}\dx{t}
 + \iint_\OmegaT \prt*{\int_0^1\curlyH_\infty(y,x,t)\dx{y}} \ v_\infty(x,t)\dx{x}\dx{t}\\
 &\qquad=\iint_\OmegaT \partialt{\varrho_\infty} v_\infty \dx{x}\dx{t}.
\end{split}
\end{equation*}
Hence Eq.~\eqref{eq: es to 0} is proven. As a consequence, Eq.~\eqref{eq: important inequality} concludes the proof.
\end{proof}

\noindent
Having proved the strong compactness of $\nabla v_\gamma$, we can finally recover the strong compactness of the pressure itself, by simply applying the Poincar\'e inequality, using the fact that $\Omega$ has been chosen large enough such that the pressure satisfies Dirichlet boundary conditions. 
\begin{corollary}[Strong compactness of $p_\gamma$]\label{coro: strong c p}
Up to the extraction of a subsequence, we have
$$p_\gamma \rightarrow p_\infty, \quad \text{ strongly in } \ L^2(\OmegaT).$$ 
\end{corollary}
\begin{proof}
Since we assumed the solutions to be compactly supported for all times $0\leq t\leq T$, by Lemma~\ref{lemma: strong convergence grad vgamma} and Poincar\'e's inequality we infer the strong compactness of $v_\gamma$ in $L^2(\OmegaT).$ Finally, since $p_\gamma= v_\gamma^{\gamma/(\gamma+1)}$ and $p_\infty=v_\infty$, the proof is completed.
\end{proof}

\noindent
Thanks to this result, we can finally identify the limit of the reaction term, \ie the following equality holds almost everywhere in $[0,1]\times \OmegaT$
\begin{equation}
    \label{eq: H}
   \curlyH_\infty(y,x,t)= n_\infty(y,x,t) R(y,p_\infty(x,t)).
\end{equation}
Thanks to the strong compactness of the pressure gradient, we can pass to the limit in Eq.~\eqref{eq: weak form n gamma} to obtain Eq.~\eqref{limit weak equation for n(y)}.

Finally, to complete the proof of Theorem~\ref{thm: complem relation}, we show that the complementarity relation \eqref{eq: p infty pheno} holds true. Let us multiply Eq.~\eqref{eq: new density} by $v_\gamma$ to get
\begin{equation*} 
   \frac{1}{\gamma+2}   \partialt{\varrho_\gamma^{\gamma+2}} = \frac{\gamma}{\gamma+1} v_\gamma \Delta v_\gamma + v_\gamma \int_0^1 n_\gamma R(y,p_\gamma) \dx{y}.
\end{equation*}
As already proven, $v_\gamma, p_\gamma$ and $\nabla v_\gamma$ are strongly compact in $L^2(\OmegaT)$. Therefore, passing to the limit $\gamma\to \infty$ we obtain 
\begin{equation*}
    v_\infty \prt*{\Delta v_\infty + \int_0^1 n_\infty(y) R(y, p_\infty)\dx{y}}=0, \; \; \text{ in } \curlyD'(\Omega\times(0,\infty)),
\end{equation*}
which concludes the proof.

\section{Additional regularity estimates}
\label{sec: add estimates}
Here we prove the additional regularity results on the pressure $p_\gamma$, namely the uniform bound on the pressure gradient in $L^4_{x,t}$. Before giving the proofs of the main results, let us remark that, unlike other works on cross-diffusion systems, we do not account for cross-reaction, but we expect the argument to go through without major changes. 
\begin{remark}
    Let us remark that unlike \cite{GPS, BPPS} in this paper we do not account for \textit{cross}-reaction. However, the proofs of Theorems~\ref{thm: L4 pheno}-\ref{remark: L4} can be extended without major changes to the cross-reaction structured counterpart, namely
    $$\partialt n(y,x,t) - \nabla \cdot(n(y,x,t) \nabla p(x,t)) = \int_0^1 n{(\eta,x,t)} R(\eta, y, p{(x,t)})\dx{\eta}.$$
    Here, the reaction term accounts for the fact that cells can mutate and change their phenotype.
\end{remark}

\begin{proof}[Proof of Theorem~\ref{thm: L4 pheno}]
First of all, let us recall that $\curlyR = \int_0^1 \sigma(\eta)R(\eta,p)\dx{\eta}$, hence $\partial_p \curlyR \leq 0$.
We multiply Eq.~\eqref{eq: pressure pheno} by $-p^\alpha (\Delta p + \curlyR)$ to obtain
\begin{equation}\label{eq: beginning L4alpha}
  -  p^\alpha\partialt p(\Delta p +\curlyR) = - \gamma p^{\alpha+1}(\Delta p + \curlyR)^2 - p^\alpha |\nabla p|^2(\Delta p + \curlyR).
\end{equation}
Now we integrate in space and we split the left-hand side treating each term individually.
\begin{align*}
  -\int_\Omega p^\alpha \partialt p \Delta p \dx{x} &= \frac 12 \int_\Omega p^\alpha \partialt{} |\nabla p|^2\dx{x} + \alpha \int_\Omega p^{\alpha-1} \partialt p |\nabla p|^2 \dx{x}\\[0.3em]
  &= \frac 12 \ddt \int_\Omega p^\alpha |\nabla p|^2 \dx{x}+ \frac \alpha 2 \int_\Omega p^{\alpha-1} \partialt p |\nabla p|^2\dx{x}\\[0.3em]
  &= \frac 12 \ddt \int_\Omega p^\alpha |\nabla p|^2\dx{x} + \frac{\alpha \gamma}{2} \int_\Omega p^\alpha (\Delta p + \curlyR) |\nabla p|^2\dx{x} + \frac{\alpha}{2} \int_\Omega p^{\alpha-1}|\nabla p|^4\dx{x}.
\end{align*}
Let us define the following function
\begin{equation*}
    \overline{\curlyR}(p,\sigma)=\int_0^p q^\alpha \curlyR(q,\sigma) \dx{q} .
\end{equation*}
It immediately follows
\begin{equation*}
     p^\alpha \partialt p \curlyR = \partialt{\overline{\curlyR}} - \int_0^1 \prt*{\int_0^p q^\alpha R(\eta,q)\dx{q}} \partial_t \sigma \dx{\eta}.
\end{equation*}
Now using the equation on the fraction density $\sigma$, Eq.~\eqref{eq: fraction density}, we have
\begin{align*}
    -\int_\Omega p^\alpha  \partialt p\curlyR \dx{x} = -\ddt& \int_\Omega \overline{\curlyR} \dx{x}+ \int_\Omega \int_0^1 \prt*{\int_0^p q^\alpha R(\eta,q)\dx{q}} \nabla \sigma\cdot\nabla p \dx{\eta}\dx{x} \\[0.4em]
    &+ \int_\Omega \int_0^1 \prt*{\int_0^p q^\alpha R(\eta,q)\dx{q}} \prt*{R(\eta,p) - \curlyR(p) }\sigma\dx{\eta}\dx{x}\\
   =-\ddt& \int_\Omega \overline{\curlyR}\dx{x} + \int_\Omega \int_0^1 \prt*{\int_0^p q^\alpha R(\eta,q)\dx{q}} \nabla \sigma\cdot\nabla p \dx{\eta}\dx{x} + \textit{Bdd},
\end{align*}
where we use $\textit{Bdd}$ to denote the bounded term
\begin{equation*}
\int_\Omega \int_0^1 \prt*{\int_0^p q^\alpha R(\eta,q)\dx{q}} \prt*{R(\eta,p) - \curlyR}\sigma\dx{\eta}\dx{x} \leq \frac{C}{\alpha+1} \int_\Omega p^{\alpha+1}\dx{x}\leq C \|p\|^2_{L^2},
\end{equation*}
where $C$ is a positive constant that depends on $\|\curlyR\|_\infty$.
Now let us come back to Eq.~\eqref{eq: beginning L4alpha} and integrate on $\Omega$
\begin{equation}
\label{eq: second step L4}
\begin{split}
\frac{\alpha}{2} \int_\Omega p^{\alpha-1}|\nabla p|^4 &\dx{x} +\gamma\int_\Omega p^{\alpha+1}(\Delta p +\curlyR)^2\dx{x} \\[0.3em]
& =-\prt*{1 +\frac{\alpha \gamma}{2}} \int_\Omega p^\alpha (\Delta p + \curlyR) |\nabla p|^2\dx{x}
+\ddt\int_\Omega \prt*{\overline{\curlyR} - p^\alpha \frac{|\nabla p|^2}{2}}\dx{x}\\[0.3em]
&\quad\qquad-\underbrace{\int_\Omega \int_0^1 \prt*{\int_0^p q^\alpha R(\eta,q)\dx{q}} \nabla \sigma\cdot\nabla p \dx{\eta}\dx{x}}_{\curlyA} -\textit{Bdd}.
\end{split}
\end{equation}
Let us integrate by parts the term $\curlyA$. We obtain
\begin{align*}
-\curlyA=&- \int_0^1  \int_\Omega\prt*{\int_0^p q^\alpha R(\eta,q)\dx{q}} \nabla \sigma\cdot\nabla p \dx{\eta} \dx{x}\\[0.3em]
=& \int_\Omega p^\alpha |\nabla p|^2 \prt*{\int_0^1 R(\eta,p)\sigma\dx{\eta}}\dx{x}
+ \int_0^1 \int_\Omega \prt*{\int_0^p q^\alpha R(\eta,q)\dx{q} } \sigma \Delta p \dx{\eta}\dx{x}\\
    \leq & \|\curlyR\|_\infty p_M^\alpha \int_\Omega |\nabla p|^2\dx{x} +\frac 12 \int_\Omega \dfrac{\prt*{\int_0^1\left(\int_0^p q^\alpha R(\eta,q)\dx{q}\right) \sigma \dx{\eta}}^2}{p^{\alpha+1}}\dx{x} + \frac 12 \int_\Omega p^{\alpha+1} |\Delta p|^2\dx{x}, 
\end{align*}
where in the last line we used Fubini's Theorem and Young's inequality.
Since by assumption both $R(y,p)$ and $\partial_p R(y,p)$ are bounded, the second term in the right-hand side is bounded. 

Combining the estimate on the term $-\curlyA$ with Eq.~\eqref{eq: second step L4} and integrating in time, we obtain
\begin{equation}\label{eq: almost}
\begin{split}
    \frac \alpha 2 \iint_{\Omega_T} p^{\alpha-1}&|\nabla p|^4 \dx{x}\dx{t}+ \gamma  \iint_{\Omega_T} p^{\alpha+1}(\Delta p +\curlyR)^2\dx{x}\dx{t} \\[0.7em]
    \leq &-\prt*{1+\frac{\alpha\gamma}{2}} \underbrace{\iint_{\Omega_T} p^\alpha(\Delta p +\curlyR)|\nabla p|^2\dx{x}\dx{t}}_{\curlyB}+ \int_\Omega \overline{\curlyR}(T)\dx{x} 
    \\[0.5em] &+ \int_\Omega (p_0)^\alpha\frac{|\nabla p_0|^2}{2}\dx{x}
    + \frac 12  \iint_{\Omega_T}\! p^{\alpha+1} |\Delta p|^2 \dx{x}\dx{t}+ \textit{Bdd},
    \end{split}
\end{equation}
where $\textit{Bdd}$ now includes other bounded quantities.
Now it remains to treat the term $\curlyB$. Let us point out here that we cannot estimate it in the same way as in \cite{MePeQu}, since the authors make use of a lower bound of the quantity $\Delta p +\curlyR$, \ie the $L^\infty$-Aronson-B\'enilan estimate, which does not hold for a multi-species system like the one at hand.
For this reason, we deal with the term $\curlyB$ by splitting it into two parts. The one coming from the source term is easier to estimate, since it can be bounded in the following way
\begin{equation}\label{eq: simple bound}
\iint_{\Omega_T}\! p^\alpha \curlyR |\nabla p|^2 \dx{x}\dx{t}\leq p_M^\alpha\|\curlyR\|_\infty \|\nabla p\|_2^2\leq \max(1,p_M)\|\curlyR\|_\infty \|\nabla p\|_2^2.
\end{equation}
The term with $\Delta p$ is instead more involved. We refer the reader to \cite{DP} for the same method applied to the case of one species and $\alpha=0$.
From now on, for the sake of simplicity, we only compute the integral in space. Integrating by parts twice we have
\begin{align}\nonumber
    \int_\Omega p^\alpha \Delta p |\nabla p|^2 \dx{x}=&\int_\Omega\Delta(p^\alpha |\nabla p|^2) p\dx{x}\\[0.3em]
    =&\int_\Omega \Delta p^\alpha |\nabla p|^2 p\dx{x} + \! 2 \alpha\! \int_\Omega \nabla p \cdot \nabla (|\nabla p|^2) p^{\alpha}\dx{x} +\!\int_\Omega p^{\alpha+1}\Delta (|\nabla p|^2)\dx{x}.
    \label{eq: dissipative term}
\end{align}
Computing the sum of the first two terms of the right-hand side, we find
\begin{align*}
  \int_\Omega \Delta p^\alpha &|\nabla p|^2 p \dx{x}+ 2 \alpha \int_\Omega \nabla p \cdot \nabla (|\nabla p|^2) p^{\alpha}\dx{x} \\
  =& \alpha(\alpha-1)\int_\Omega p^{\alpha-1}|\nabla p|^4\dx{x} + \alpha \int_\Omega p^\alpha \Delta p |\nabla p|^2\dx{x} - 2\alpha\int_\Omega p^\alpha \Delta p |\nabla p|^2 \dx{x}- 2 \alpha^2 \int_\Omega p^{\alpha-1}|\nabla p|^4\dx{x}\\
  =& -\alpha(\alpha+1)\int_\Omega p^{\alpha-1}|\nabla p|^4 \dx{x}-\alpha  \int_\Omega p^\alpha \Delta p |\nabla p|^2\dx{x} ,
\end{align*}
where we used integration by parts on the second term.

We compute the last term in Eq.~\eqref{eq: dissipative term} as follows
\begin{align*}
    \int_\Omega p^{\alpha+1}\Delta (|\nabla p|^2)\dx{x} &= 2\int_\Omega p^{\alpha+1}\nabla p \cdot \nabla (\Delta p)\dx{x} +2\int_\Omega p^{\alpha+1}(D_{i,j}^2 p)^2\dx{x}\\
    &=-2(\alpha+1)\int_\Omega p^\alpha |\nabla p|^2 \Delta p\dx{x} - 2 \int_\Omega p^{\alpha+1}|\Delta p|^2 \dx{x}+ 2 \int_\Omega p^{\alpha+1}(D_{i,j}^2 p)^2\dx{x},
\end{align*}
where in the last equality we used integration by parts and we denoted $(D_{i,j}^2 p)^2 =\sum_{i,j} (\partial_{i,j}^2 p)^2$.
By consequence, Eq.~\eqref{eq: dissipative term} now reads
\begin{align*}
    \int_\Omega p^\alpha \Delta p |\nabla p|^2 \dx{x}=-\alpha&(\alpha+1)\int_\Omega p^{\alpha-1}|\nabla p|^4 \dx{x}-(3\alpha+2) \int_\Omega p^\alpha \Delta p |\nabla p|^2\dx{x}\\
     &- 2 \int_\Omega p^{\alpha+1}|\Delta p|^2 \dx{x}+ 2 \int_\Omega p^{\alpha+1}(D_{i,j}^2 p)^2\dx{x},
\end{align*}
and thus
\begin{equation}\label{eq: dissipative final}
\begin{split}
      \int_\Omega p^\alpha \Delta p |\nabla p|^2 \dx{x}= &-\frac \alpha 3 \int_\Omega p^{\alpha-1}|\nabla p|^4 \dx{x}-\frac{2}{3(\alpha+1)}  \int_\Omega p^{\alpha+1}|\Delta p|^2\dx{x}\\[0.3em]
    &+ \frac{2}{3(\alpha+1)}\int_\Omega p^{\alpha+1}(D_{i,j}^2 p)^2\dx{x}.
\end{split}
\end{equation}
Using Eq.~\eqref{eq: dissipative final} in Eq.~\eqref{eq: almost}, we finally find
\begin{equation*} 
\begin{split}
    \frac \alpha 2  \iint_{\Omega_T}\!&p^{\alpha-1}|\nabla p|^4\dx{x}\dx{t} + \gamma  \iint_{\Omega_T}\! p^{\alpha+1}(\Delta p +\curlyR)^2\dx{x}\dx{t} +\frac{2+\alpha \gamma}{3(\alpha+1)}\iint_{\Omega_T}p^{\alpha+1} (D^2_{i,j} p)^2\dx{x}\dx{t}\\[0.4em]
    \leq & \ \frac \alpha 3 \prt*{1+\frac{\alpha\gamma}{2}}\iint_{\Omega_T}\!p^{\alpha-1}|\nabla p|^4\dx{x}\dx{t} +\prt*{\frac{2+\alpha\gamma}{3(\alpha+1)}+\frac 12}\iint_{\Omega_T}\!p^{\alpha+1}|\Delta p|^2\dx{x}\dx{t} + \textit{Bdd},
    \end{split}
\end{equation*}
where $\textit{Bdd}$ includes also the bound in Eq.~\eqref{eq: simple bound}.
By Young's inequality, we have
\begin{equation*}
   \iint_{\Omega_T}\!p^{\alpha+1}|\Delta p|^2 \dx{x}\dx{t}\leq \frac 32 \iint_{\Omega_T}\!p^{\alpha+1}|\Delta p+\curlyR|^2\dx{x}\dx{t} + 3 \iint_{\Omega_T}\!p^{\alpha+1}|\curlyR|^2\dx{x}\dx{t}.
\end{equation*}
Then, we finally have 
\begin{equation}\label{eq: complete estimate for L4}
\begin{split}
      \kappa(\alpha) \iint_{\Omega_T}\!p^{\alpha-1}|\nabla p|^4\dx{x}\dx{t} +\left( \gamma-\frac 32 \right) &\iint_{\Omega_T}\! p^{\alpha+1}(\Delta p +\curlyR)^2\dx{x}\dx{t}\\[0.4em]
      &+\frac{2+\alpha \gamma}{3(\alpha+1)}\iint_{\Omega_T}\!p^{\alpha+1} (D^2_{i,j} p)^2\dx{x}\dx{t} \leq C(T), 
\end{split}
\end{equation}
with $\kappa(\alpha):=\frac{\alpha}{6}(1-\alpha\gamma)$. Since we assumed $0< \alpha< \frac{1}{\gamma}$, this concludes the proof.
\end{proof}

\noindent
For $\alpha=0$, Eq.~\eqref{eq: complete estimate for L4} proven above immediately implies a bound on the pressure gradient which is uniform with respect to $\gamma$. This bound was also investigated in \cite{DP}, where the authors prove its sharpness.

\begin{proof}[Proof of Theorem~\ref{remark: L4}]
Let us take $\alpha =0$ in Eq.~\eqref{eq: complete estimate for L4}. Then, we infer the following bounds
\begin{equation*}
     \iint_{\Omega_T}\! p(\Delta p +\curlyR)^2 \dx{x}\dx{t}\leq C(T), \qquad \iint_{\Omega_T}\!p  (D^2_{i,j} p)^2\dx{x}\dx{t}\leq C(T),
\end{equation*}
and both hold uniformly with respect to $\gamma$. Since both $p$ and $\curlyR$ are uniformly bounded in $L^\infty$, this implies
\begin{equation*}
     \iint_{\Omega_T}\! p^2|\Delta p|^2 \dx{x}\dx{t}\leq C(T), \qquad \iint_{\Omega_T}\!p^2  (D^2_{i,j} p)^2\dx{x}\dx{t}\leq C(T).
\end{equation*} 
Using integration by parts, it follows immediately that the boundedness of these two terms implies $\nabla p \in~L^4(\Omega_T)$.
\end{proof}

\section*{Acknowledgements}
This project has received funding from the European Union's Horizon 2020 research and innovation program under the Marie Skłodowska-Curie (grant agreement No 754362).\\
The author would like to thank Beno\^it Perthame for the fruitful discussions throughout the preparation of this paper.
The author would also like to thank Tomasz D\k{e}biec and Matt Jacobs for their interest in this work and their valuable comments and suggestions.

\begin{appendices}
 
\section{Compensated compactness}\label{app: cc}
\begin{theorem}\label{thm: comp comp}
Let $u_\gamma, w_\gamma \in L^\infty(0,T;L^2(\Omega))$, and let $u_\infty, w_\infty$ be the $L^2$-weak limits of $u_\gamma, w_\gamma$ as $\gamma\to\infty$, respectively. We assume that
\begin{equation*}
    \partialt{u_\gamma} \in L^2(0,T; H^{-1}(\Omega)), \qquad w_\gamma \in L^2(0,T; H^1(\Omega)).
\end{equation*}
Then, up to a subsequence, we have
\begin{equation*}
    \iint_{\Omega_T} u_\gamma w_\gamma \varphi \dx{x}\dx{t} \xrightarrow{\gamma\to\infty} \iint_{\Omega_T} u_\infty w_\infty \varphi\dx{x}\dx{t},
\end{equation*}
for all $\varphi\in C(0,T; C^1(\Omega))$.
\end{theorem}
\begin{proof}
Let $\psi_\es (x):= \frac{1}{\es^d} \psi(\frac{x}{\es})$ for $x\in\R^d$ and $\zeta_\sigma (t):= \frac{1}{\sigma} \zeta(t)$, for $t>0$ be smooth mollifiers.
Then, we compute
\begin{align*}
    \iint_{\Omega_T} u_\gamma w_\gamma \varphi \dx{x}\dx{t}= & \iint_{\Omega_T} u_\gamma (w_\gamma \varphi - (w_\gamma \varphi) \star_x \psi_\es) \dx{x}\dx{t} + \iint_{\Omega_T} u_\gamma (w_\gamma \varphi) \star_x \psi_\es \dx{x}\dx{t}\\[0.3em]
    = & \iint_{\Omega_T}\left(\int_{\R^d} (w_\gamma(x)\varphi(x) - w_\gamma(x-\es z)\varphi(x-\es z)) \psi(z)\dx{z}\right)  u_\gamma\dx{x}\dx{t} \\[0.3em]
    \qquad &+ \iint_{\Omega_T} (u_\gamma- u_\gamma \star_t \zeta_\sigma) (w_\gamma\varphi) \star_x \psi_\es \dx{x}\dx{t}  + \iint_{\Omega_T} (u_\gamma \star_t \zeta_\sigma) (w_\gamma \varphi)\star_x \psi_\es \dx{x}\dx{t} .
\end{align*}
Passing to the limit subsequently in $\varepsilon\to 0, \gamma \to \infty$, and $\delta \to 0$, we have
\begin{equation*}
   \iint_{\Omega_T} (u_\gamma \star_t \zeta_\sigma) (w_\gamma \varphi)\star_x \psi_\es \dx{x}\dx{t} \to \iint_{\Omega_T} u_\infty w_\infty \varphi \dx{x}\dx{t}.
\end{equation*}
It now remains to prove that the other terms converge to zero as $\es\to 0$ and $\sigma\to 0$. By the Fr\'echet-Kolmogorov theorem, we know that 
\begin{align*}
  \int_{\Omega} |(w_\gamma \varphi) (x) -& (w_\gamma \varphi)(x+k)|^2 \dx{x} \\[0.5em]
  &\leq  \int_{\Omega} |w_\gamma(x)( \varphi (x) -\varphi(x+k))|^2\dx{x} + \int_\Omega |\varphi(x+k)(w_\gamma(x)-w_\gamma(x+k)|^2 \dx{x}\\[0.5em]
  &\leq \omega(|k|),
\end{align*}          
where $\omega(|k|)\to 0$ as $k\to 0$. Hence
\begin{align*}
    \iint_{\Omega_T}&\left(\int_{\R^d} (w_\gamma(x)\varphi(x) - w_\gamma(x-\es z)\varphi(x-\es z))\psi(z)\dx{z}\right) u_\gamma(x,t)\dx{x}\dx{t} \\[0.3em]
    &= \int_0^T\int_{\R^d}\left(\int_{\Omega} (w_\gamma(x)\varphi(x) - w_\gamma(x-\es z)\varphi(x-\es z)) u_\gamma(x,t) \dx{x} \right) \psi(z)\dx{z}\dx{t}\\[0.3em]
    &\leq \int_0^T\int_{\R^d} (\omega(\es|z|))^{1/2} \|u_\gamma(t)\|_{L^2(\Omega)} \psi(z)\dx{z}\dx{t} \to 0.
\end{align*}
Now we treat the last term. For the sake of brevity, let us denote $(w_\gamma \varphi)_\es := (w_\gamma \varphi)\star_x \psi_\es$
\begin{align*}
    \iint_{\Omega_T} (u_\gamma- u_\gamma \star_t \zeta_\sigma) (w \varphi)_\es \dx{x}\dx{t} &=  \iint_{\Omega_T} \left( \int_\R (u_\gamma(t)- u_\gamma(t-\sigma s))\zeta(s)\dx{s}\right) (w_\gamma \varphi)_\es\dx{x}\dx{t}\\
    &=  \iint_{\Omega_T} \left[ \int_\R \left(\int_{t-\sigma s}^t \partialt{u_\gamma(\tau)}\dx{\tau}\right)\right] (w_\gamma \varphi)_\es\dx{x}\dx{t}\\
    &= \int_\R\zeta(s) \left(\int_0^T \int_{t-\sigma s}^t \int_{\Omega}  \partialt{u_\gamma(\tau)} (w_\gamma \varphi)_\es\dx{x}\dx{\tau}\dx{t}\right)\dx{s}\\
   &\leq \int_\R \zeta(s) \int_0^T \left(\int_{t-\sigma s}^t \left\|\partialt{u_\gamma(\tau)}\right\|_{H^{-1}(\Omega)}\dx{\tau}\right) \|(w_\gamma \varphi)_\es\|_{H^1(\Omega)} \dx{t}\dx{s}\\
   & \leq C \sigma \int_\R \zeta(s) |s| \int_0^T \|(w_\gamma \varphi)_\es\|_{H^1(\Omega)} \dx{t}\dx{s} \leq C \sigma \to 0,
\end{align*}
as $\sigma \to 0.$

\end{proof}
\section{Convergence of the reaction terms}\label{app: N}
\begin{lemma}
Equations \eqref{eq: conv curlyR} and \eqref{eq: conv n R} hold, \ie
\begin{align*}
    \curlyR_\es \rightharpoonup \curlyR \quad &\text{ weak$^*$ in } L^\infty(Q_T),\\[0.4em]
\label{eq: conv n R}
    n_\es R(y,p_\es) \rightharpoonup n R(y,p) \quad &\text{ weak$^*$ in } L^\infty([0,1]\times Q_T). 
\end{align*}
\end{lemma}
\begin{proof}
By the Stone-Weierstrass theorem we know that, for any $\delta>0$, there exists $N>0$ and $\{a_i\}_{i=1}^N$ and $\{G_i\}_{i=1}^N$ such that
\begin{equation}
    \left\|R(y,p_\es)-\sum_{i=1}^N a_i(y) G_i(p_\es)\right\|_{L^\infty} \leq \delta.
\end{equation}
Let $\varphi\in L^1(Q_T)$, such that $\|\varphi\|_{L^1}=1$. Since $\sigma_\es \rightharpoonup \sigma$ weakly$^*$ in $L^\infty((0,1)\times Q_T)$ and $p_\es\rightarrow p$ strongly in $L^2(Q_T)$ as $\es\to 0$, we have
\begin{equation*}
\begin{split}
    \iint_{Q_T} \prt*{\sum_{i=1}^N \int_0^1 \sigma_\es (\eta) a_i(\eta) G_i( p_\es)\dx{\eta}} \varphi(x,t)   \dx{x}\dx{t}=&\sum_{i=1}^N \int_0^1\iint_{Q_T}  \sigma_\es (\eta) a_i(\eta) G_i( p_\es)\varphi(x,t) \dx{x}\dx{t}\dx{\eta}\\[0.4em]
\xrightharpoonup{\es\to 0} &\sum_{i=1}^N \int_0^1\iint_{Q_T}  \sigma (\eta) a_i(\eta) G_i( p)\varphi(x,t) \dx{x}\dx{t}\dx{\eta}.
\end{split}
\end{equation*} 
Therefore, there exists $\es_0$ such that for all $\es<\es_0$ 
\begin{equation}
    \iint_{Q_T}\prt*{ \sum_{i=1}^N \int_0^1 \sigma_\es(\eta) a_i(\eta)G_i(p_\es)\dx{\eta} - \sum_{i=1}^N \int_0^1 \sigma(\eta) a_i(\eta)G_i(p)\dx{\eta}}\varphi \dx{x}\dx{t} \leq \delta.
\end{equation}
We compute
\begin{equation*}
\begin{split}
  &\iint_{Q_T} \prt*{ \int_0^1 \sigma_\es (\eta) R(\eta,p_\es)\dx{\eta}- \int_0^1 \sigma (\eta) R(\eta,p)\dx{\eta}} \varphi(x,t) \dx{x}\dx{t}\\
  \leq & 
  \left\|\int_0^1 \sigma_\es (\eta) R(\eta,p_\es)\dx{\eta} - \sum_{i=1}^N \int_0^1 \sigma_\es(\eta)a_i(\eta)G_i(p_\es)\dx{\eta}\right\|_{L^\infty} \|\varphi\|_{L^1}\\
  &+ \iint_{Q_T}\prt*{ \sum_{i=1}^N \int_0^1 \sigma_\es(\eta) a_i(\eta)G_i(p_\es)\dx{\eta} - \sum_{i=1}^N \int_0^1 \sigma(\eta) a_i(\eta)G_i(p)\dx{\eta}}\varphi \dx{x}\dx{t}\\
  &+ \left\|\sum_{i=1}^N \int_0^1 \sigma(\eta) a_i(\eta)G_i(p)\dx{\eta} - \int_0^1 \sigma (\eta) R(\eta,p)\dx{\eta} \right\|_{L^\infty} \|\varphi\|_{L^1} \leq 3\delta,
\end{split}
\end{equation*} 
for $\es\leq\es_0$. Since $\delta$ was chosen arbitrarily, we conclude that
\begin{align*}
\curlyR_\es := \int_0^1 \sigma_\es (\eta) R(\eta, p_\es)\dx{\eta} \rightharpoonup \int_0^1 \sigma (\eta) R(\eta, p)\dx{\eta}:=\curlyR, \quad \text{ weakly}^* \text{ in } L^\infty(Q_T).
\end{align*}
\ie \eqref{eq: conv curlyR} is proven.
By an analogous argument, we have
$$n_\es R(y,p_\es) \rightharpoonup n R(y,p), \quad \text{ weakly}^* \text{ in } L^\infty((0,1)\times Q_T),$$
and this concludes the proof of \eqref{eq: conv n R}.
\end{proof}

\end{appendices}

\bibliographystyle{abbrv}
\bibliography{bib_final}

\end{document}